\documentclass[preprint,11pt]{elsarticle}

\usepackage{amsfonts, amsmath, amscd}
\usepackage[psamsfonts]{amssymb}

\usepackage{amssymb}

\usepackage{pb-diagram}

\usepackage[all,cmtip]{xy}

\usepackage[usenames]{color}

\newtheorem{theorem}{Theorem}[section]
\newtheorem{lemma}[theorem]{Lemma}
\newtheorem{corollary}[theorem]{Corollary}

\newtheorem{remark}[theorem]{Remark}
\newtheorem{proposition}[theorem]{Proposition}
\newtheorem{definition}[theorem]{Definition}
\newtheorem{example}[theorem]{Example}

\newtheorem{notation}[theorem]{Notation}

\newproof{proof}{Proof}

\numberwithin{equation}{section}
\numberwithin{theorem}{section}
\newcommand{\e}{\varepsilon}
\newcommand{\w}{\omega}

\newcommand{\IR}{\mathbb{R}}

\newcommand{\ff}{\mathbb{F}}

\newcommand{\IF}{\mathbb{F}}

\newcommand{\xxx}{\mathbf{x}}

\newcommand{\TTT}{\mathcal{T}}

\newcommand{\EE}{{\mathcal E}}

\newcommand{\Nn}{\mathcal{N}}

\newcommand{\LL}{\mathcal{L}}

\newcommand{\A}{\mathcal{A}}
\newcommand{\supp}{\mathrm{supp}}

\newcommand{\cl}{\mathrm{cl}}
\newcommand{\Ra}{\Rightarrow}

\newcommand{\LRa}{\Leftrightarrow}
\newcommand{\cacx}{\overline{\mathrm{acx}}}

\newcommand{\Bo}{\mathsf{Bo}}

\newcommand{\Id}{\mathsf{id}}

\newcommand{\spn}{\mathsf{span}}
\newcommand{\cspn}{\overline{\mathsf{span}}}

\newcommand{\SM}{{\setminus}}

\begin{document}

\begin{frontmatter}

\title{Limited type subsets  of locally convex spaces}

\author{Saak Gabriyelyan}%\tnotetext[label1]{The first named author was partially supported by Israel Science Foundation grant 1/12.}
\ead{saak@math.bgu.ac.il}
\address{Department of Mathematics, Ben-Gurion University of the Negev, Beer-Sheva, P.O. 653, Israel}

\begin{abstract}
Let $1\leq p\leq q\leq\infty.$ Being motivated by the classical notions of limited, $p$-limited and coarse $p$-limited subsets of a Banach space, we introduce and study $(p,q)$-limited subsets and their equicontinuous versions and coarse $p$-limited subsets of an arbitrary locally convex space $E$. Operator characterizations of these classes are given. We compare these classes with the classes of bounded, (pre)compact, weakly (pre)compact  and relatively weakly sequentially (pre)compact sets. If $E$ is a Banach space, we show that the class of coarse $1$-limited subsets of $E$ coincides with the class of $(1,\infty)$-limited sets, and if $1<p<\infty$, then the class of coarse $p$-limited sets in $E$ coincides with the class of $p$-$(V^\ast)$ sets of Pe{\l}czy\'{n}ski. We also generalize a known theorem of Grothendieck.
\end{abstract}


\end{frontmatter}

%%%%%%%%%%%%%%%%%%%%%%%%%%%%%%%%%%%%%%%%%%%
%%%%%%%%%%%%%%%%%%%%%%%%%%%%%%%%%%%%%%%%%%%
%%%%%%%%%%%%%%%%%%%%%%%%%%%%%%%%%%%%%%%%%%%
%%%%%%%%%%%%%%%%%%%%%%%%%%%%%%%%%%%%%%%%%%%
%%%%%%%%%%%%%%%%%%%%%%%%%%%%%%%%%%%%%%%%%%%

\section{Introduction}

%%%%%%%%%%%%%%%%%%%%%%%%%%%%%%%%%%%%%%%%%%%
%%%%%%%%%%%%%%%%%%%%%%%%%%%%%%%%%%%%%%%%%%%
%%%%%%%%%%%%%%%%%%%%%%%%%%%%%%%%%%%%%%%%%%%
%%%%%%%%%%%%%%%%%%%%%%%%%%%%%%%%%%%%%%%%%%%
%%%%%%%%%%%%%%%%%%%%%%%%%%%%%%%%%%%%%%%%%%%

Let $E$ be a locally convex space (lcs for short), and let $E'$ denote the topological dual of $E$.
For a bounded subset $A\subseteq E$ and a functional $\chi\in E'$, we put
\[
\|\chi\|_A:= \sup\big\{ |\chi(x)|:x\in A\cup\{0\}\big\}.
\]

\begin{definition} \label{def:limited-Banach-BD} {\em
A bounded subset $A$ of a Banach space $E$ is called {\em limited} if each weak$^\ast$ null sequence $\{ \chi_n\}_{n\in\w}$ in $E'$ converges to zero uniformly on $A$, that is $\lim_{n\to\infty} \|\chi_n\|_A =0$. Denote by $\mathsf{L}(E)$ the family of all limited subsets of $E$.\qed}
\end{definition}

Limited sets in Banach spaces were systematically studied by Bourgain and Diestel  \cite{BourDies}, see also Schlumprecht \cite{Schlumprecht-Ph}. Among other things they proved the following result (all relevant definitions are given in Section \ref{sec:pre}).
\begin{theorem}[\cite{BourDies}] \label{t:limited-BD}
Let $E$ be a Banach space. Then:
\begin{enumerate}
\item[{\rm(i)}] $\mathsf{L}(E)$ is closed  under taking subsets, finite sums and absolutely convex hulls;
\item[{\rm(ii)}] if $E$ contains no copy of $\ell_1$, then each $A\in\mathsf{L}(E)$ is relatively weakly compact;
\item[{\rm(iii)}] every $A\in\mathsf{L}(E)$ is weakly sequentially precompact;
\item[{\rm(iv)}] if $E$ is separable or reflexive, then each $A\in\mathsf{L}(E)$ is relatively compact.
\end{enumerate}
\end{theorem}

Let $E$ and $H$ be locally convex spaces. Denote by $\LL(E,H)$ the family of all operators from $E$ to $H$. If $p\in[1,\infty]$,  a sequence $\{x_n\}_{n\in\w}$ in $E$ is called {\em weakly $p$-summable} if  for every $\chi\in E'$ it follows that $(\langle\chi,x_n\rangle)\in \ell_p$ if $p\in[1,\infty)$, or $(\langle\chi,x_n\rangle)\in c_0$ if $p=\infty$. The family $\ell_p^w(E)$ (or $c_0^w(E)$ if $p=\infty$) of all weakly $p$-summable sequences in $E$ is a vector space which admits a natural locally convex vector topology such that it is complete if so is $E$, for details see Section 19.4 in \cite{Jar} or Section 4 in \cite{Gab-Pel}.

Let $p\in[1,\infty]$, and let $X$ and $Y$ be Banach spaces.  Generalizing the notion of completely continuous operators   Castillo and S\'{a}nchez  defined in \cite{CS} an operator $T:X\to Y$ to be {\em $p$-convergent} if $T$ sends weakly $p$-summable sequences of $X$ into norm null-sequences of $Y$. %In this  fundamental article  they  defined the Banach space $X$ to  have the {\em Dunford--Pettis property of order $p$} (the {\em $DP_p$ property}) if for each Banach space $Y$, every weakly compact operator $T:X\to Y$ is $p$-convergent.
The influential article  of Castillo and S\'{a}nchez \cite{CS} inspired an intensive study of $p$-versions of numerous geometrical properties of Banach spaces. In particular, the following $p$-versions of limitedness were introduced by Karn and Sinha \cite{KarnSinha} and Galindo and Miranda \cite{GalMir}.

\begin{definition} \label{def:small-bounded-p} {\em
Let $p\in[1,\infty]$, and let $X$ be a Banach space. A  bounded subset $A$ of $X$ is called
\begin{enumerate}
\item[{\rm(i)}] a {\em $p$-limited set} if
\[
\big(\sup_{a\in A} |\langle\chi_n,a\rangle|\big)\in \ell_p \;\; (\mbox{or } \big(\sup_{a\in A} |\langle\chi_n,a\rangle|\big)\in c_0 \; \mbox{ if } p=\infty)
\]
for every $(\chi_n)\in \ell_p^w (X^\ast)$ (or $(\chi_n)\in c_0^w(X^\ast)$ if $p=\infty$) (\cite{KarnSinha});
\item[{\rm(ii)}]  a {\em coarse $p$-limited set} if for every $T\in\LL(E,\ell_p)$ $($or $T\in\LL(E,c_0)$ if $p=\infty$$)$, the set $T(A)$ is relatively compact (\cite{GalMir}).
\end{enumerate}}
\end{definition}
It turns out that the family $\mathsf{L}_p(X)$ of all $p$-limited subsets of $X$ and the family $\mathsf{CL}_p(X)$ of all coarse $p$-limited subsets of $X$ have similar properties as $\mathsf{L}(X)$ described in Theorem \ref{t:limited-BD}, see \cite{KarnSinha} and \cite{GalMir}, respectively.

Limited sets in Fr\'{e}chet spaces were studied by Alonso \cite{Alonso}. The notion of a limited set in general {\em locally convex spaces} was introduced by Lindstr\"{o}m and Schlumprecht in \cite{Lin-Schl-lim} and independently by Banakh and Gabriyelyan in  \cite{BG-GP-lcs}. Since limited sets in the sense of \cite{Lin-Schl-lim} are defined using {\em equicontinuity}, to distinguish both notions  we called them in  \cite{BG-GP-lcs} by $\EE$-limited sets.
\begin{definition} \label{def:limited-lcs} {\em
A subset   $A$ of a locally convex space $E$ is called
\begin{enumerate}
\item[{\rm(i)}] {\em $\EE$-limited } if $\|\chi_n\|_A\to 0$ for every equicontinuous weak$^\ast$ null sequence $\{\chi_n\}_{n\in\w}$ in  $E'$ (\cite{Lin-Schl-lim});
\item[{\rm(ii)}] {\em limited }  if $\|\chi_n\|_A\to 0$ for every weak$^\ast$ null sequence $\{\chi_n\}_{n\in\w}$ in  $E'$ (\cite{BG-GP-lcs}).\qed
\end{enumerate} }
\end{definition}
\noindent It is clear that if $E$ is a $c_0$-barrelled (for example, Banach) space, then $A$ is limited  if and only if it is $\EE$-limited.

Definitions \ref{def:limited-Banach-BD}, \ref{def:small-bounded-p}  and \ref{def:limited-lcs} and the notions of $(p,q)$-$(V^\ast)$ subsets and $(p,q)$-$(EV^\ast)$ subsets of a locally convex space $E$ introduced and studied in \cite{Gab-Pel} motivate the following notions.

\begin{definition}\label{def:p-limit-coarse-set}{\em
Let $p,q\in[1,\infty]$. A non-empty subset $A$ of a separated topological vector space $E$  is called
\begin{enumerate}
\item[{\rm(i)}]  a  {\em $(p,q)$-limited set} (resp., {\em $(p,q)$-$\EE$-limited set}) if
\[
\Big(\|\chi_n\|_A\Big)\in \ell_q \; \mbox{ if $q<\infty$, } \; \mbox{ or }\;\; \|\chi_n\|_A\to 0 \; \mbox{ if $q=\infty$},
\]
for every (resp., equicontinuous) weak$^\ast$ $p$-summable sequence $\{\chi_n\}_{n\in\w}$ in  $E'$. We denote by $\mathsf{L}_{(p,q)}(E)$ and $\mathsf{EL}_{(p,q)}(E)$ the family of all $(p,q)$-limited subsets and all $(p,q)$-$\EE$-limited subsets of $E$, respectively. $(p,p)$-limited sets and $(\infty,\infty)$-limited sets will be called simply {\em $p$-limited sets} and  {\em limited sets}, respectively. %Analogously, $(p,p)$-$\EE$-limited sets and $(\infty,\infty)$-$\EE$-limited sets will be called  {\em $p$-$\EE$-limited sets} and  {\em $\EE$-limited sets}, respectively.
\item[{\rm(ii)}]  a {\em coarse $p$-limited set} if for every $T\in\LL(E,\ell_p) $ $($or $T\in\LL(E,c_0)$ if $p=\infty$), the set $T(A)$ is relatively compact. The family of all coarse $p$-limited sets is denoted by $\mathsf{CL}_p(E)$. \qed
\end{enumerate} }
\end{definition}

The purpose of the article is to study  $(p,q)$-limited subsets and  coarse $p$-limited subsets of locally convex spaces in the spirit of Theorem \ref{t:limited-BD} and the articles \cite{KarnSinha} and \cite{GalMir}.

Now we describe the content of the article. In Section \ref{sec:pre} we fix the main notions and some auxiliary results used in what follows.

In Section \ref{sec:limit-set} we study the classes  $\mathsf{L}_{(p,q)}(E)$ and $\mathsf{EL}_{(p,q)}(E)$. In Lemma \ref{l:limited-set-1} we generalize (i) of Theorem  \ref{t:limited-BD} and show that $\mathsf{L}_{(p,q)}(E)=\mathsf{EL}_{(p,q)}(E)=\{0\}$  if $q<p$. In Proposition \ref{p:product-sum-limited-set} we characterize $(p,q)$-limited subsets and $(p,q)$-$\EE$-limited subsets in products and direct sums of locally convex spaces. In Theorem  \ref{t:pq-limited-summing} we give an operator characterization of $(p,q)$-limited and $(p,q)$-$\EE$-limited subsets of the  locally convex space $E$. The following diagram easily follows from Definition \ref{def:p-limit-coarse-set} (see also (vi) of Lemma \ref{l:limited-set-1})
\[
\xymatrix{
\left.{\substack{\mbox{limited} \\ \mbox{$(p,q)$-limited}}}  \right\} \ar@{=>}[r]& \mbox{$(p,\infty)$-limited} \ar@{=>}[r]& \mbox{$(1,\infty)$-limited}.
}
\]
This diagram motivates the study of $(p,\infty)$-limited sets and $(1,\infty)$-limited sets. It is well known that any (pre)compact subset of a Banach space is limited. In Proposition \ref{p:precompact-p-limited} we generalize this useful result by showing that each precompact subset of an lcs $E$ is $(p,\infty)$-$\EE$-limited, and if in addition $E$ is $p$-barrelled, then every precompact subset of $E$ is $(p,\infty)$-limited. Consequently, each precompact subset of a $c_0$-barrelled space is limited.  In Theorem \ref{t:Cp-p-barreled} we show that every precompact subset of $C_p(X)$ (= the space $C(X)$ of all continuous functions over a Tychonoff space $X$ endowed with the pointwise topology) is  $(p,q)$-limited if and only if $X$ has no infinite functionally bounded subsets. As a corollary (see Example \ref{exa:compact-not-limited}) we obtain that the metrizable space $C_p([0,\w])$ has even compact subsets which are {\em not} limited. Being motivated by (iv) of Theorem \ref{t:limited-BD} it is natural to consider the case when every $(p,q)$-limited set is precompact. This problem is solved in Theorem \ref{t:limited-set-precompact}. In Proposition \ref{p:char-1-limited} we characterize $(1,\infty)$-limited subsets of barrelled locally convex spaces. In \cite{Grothen} (see also Theorem 3.11 of \cite{HMVZ}) Grothendieck proved that if $E$ is a Banach space, then a bounded subset $B$ of $\big(E',\mu(E',E)\big)$ is precompact if and only if it is limited. In Theorem \ref{t:weakly-mu-p-limited} we generalize this result.

In Section \ref{sec:coarse} we study coarse $p$-limited subsets of locally convex spaces. Generalizing Proposition 2 of  \cite{GalMir} we show in Lemma \ref{l:coarse-p-limited} that the family $\mathsf{CL}_p(E)$ of all coarse $p$-limited  sets in $E$ is closed under taking subsets, finite unions, closed absolutely convex hulls, and continuous linear images. In Proposition \ref{p:pV*-coarse-p-lim} we show that every $p$-limited subset of $E$ is  coarse $p$-limited (this generalizes Proposition 1 of \cite{GalMir}), and under addition assumption we prove that even every $(p,p)$-$(V^\ast)$ subset of $E$ is  coarse $p$-limited. A description of coarse $p$-limited subsets of direct products and direct sums is given in Proposition \ref{p:product-coarse-p-lim}.

%It turns out that the Grothendieck property plays an essential role in the study of limited type sets. By this reason and being motivated by \cite{DoLS}, in Section \ref{sec:Groth} we introduce and study the $(p,q)$-Grothendieck property and the notions of $(p,q)$-Grothendieck operators and $(p,q)$-Grothendieck subsets of locally convex spaces, where $1\leq p\leq q\leq\infty$. A characterization of Banach spaces with the Grothendieck property is given in Theorem 1  of \cite{Alimoh}. Generalizing this result we characterize $p$-Grothendieck locally convex spaces in Theorem \ref{t:p-Grothen}.

The following class of linear maps is defined and studied in Section 16 of \cite{Gab-Pel}.

\begin{definition} \label{def:qp-summable} {\em
Let $1\leq p\leq q\leq \infty$, and let $E$ and $L$ be locally convex spaces. A linear map $T:E\to L$ is called {\em $(q,p)$-convergent} if it sends weakly $p$-summable sequences in $E$ to strongly $q$-summable sequences in $L$. \qed}
\end{definition}

Let $1\leq p\leq q\leq \infty$. In Section 16 of \cite{Gab-Pel} we naturally extend the notion of $p$-convergent operators between Banach spaces to the general case saying that a linear map $T:E\to L$ between locally convex spaces $E$ and $L$ is {\em $(q,p)$-convergent} if it sends weakly $p$-summable sequences in $E$ to strongly $q$-summable sequences in $L$ (so $(\infty,p)$-convergent operators are exactly $p$-convergent operators). The notion of $(q,p)$-convergent operators is useful to solve the following general problem: Characterize those operators $T$ which map {\em  all  bounded } sets into $(p,q)$-limited sets (or into coarse $p$-limited  sets). If $E$ and $L$ are Banach spaces and $p=q$, a partial answer to this problem is given by Ghenciu, see Theorem 14 of \cite{Ghenciu-pGP}. In Section \ref{sec:small-bound-p-conv} we give a complete answer to this problem, see  Theorem \ref{t:bounded-to-p-lim}. The clauses (ii)-(iv) of Theorem \ref{t:limited-BD} motivate the problem of finding conditions on a space $E$ under which $(p,q)$-limited sets and coarse $p$-limited  sets have additional topological properties. For $p$-limited subsets of Banach spaces this problem was considered by Ghenciu, see Theorem 15 of \cite{Ghenciu-pGP}. In Theorem \ref{t:qp*-convergent-Limited} we essentially generalize Ghenciu's result. In Theorem \ref{t:image-l1} we characterize coarse $1$-limited sets. As a consequence of the obtained results we show in  Corollary \ref{c:Banach-p-coarse-limited} that: (1) if $p=1$, then the class of coarse $1$-limited subsets of a Banach space $E$ coincides with the class of $(1,\infty)$-limited sets, and (2) if $1<p<\infty$, then the class of coarse $p$-limited sets in $E$ coincides with the class of $p$-$(V^\ast)$ sets. It should be mentioned that $p$-$(V^\ast)$ sets in Banach spaces were defined and study by Chen,  Ch\'{a}vez-Dom\'{\i}nguez and Li in \cite{CCDL} and \cite{LCCD}.
Using the idea of the proof of (iii) of Theorem \ref{t:limited-BD}, Galindo and Miranda proved in Proposition 3 of \cite{GalMir} that if $2\leq p<\infty$, then every coarse $p$-limited set is weakly sequentially precompact. In Theorem \ref{t:coarse-p-lim-wsc} we extend this result to locally convex spaces with the Rosenthal property.

The clause (iv) of Theorem \ref{t:limited-BD} implies that each separable or reflexive Banach space has the Gelfand--Phillips property. By this reason
generalizations of this clause will be given in the forthcoming article \cite{Gab-GP}.

\section{Preliminaries results} \label{sec:pre}

%%%%%%%%%%%%%%%%%%%%%%%%%%%%%%%%%%%%%
%%%%%%%%%%%%%%%%%%%%%%%%%%%%%%%%%%%%%
%%%%%%%%%%%%%%%%%%%%%%%%%%%%%%%%%%%%%
%%%%%%%%%%%%%%%%%%%%%%%%%%%%%%%%%%%%%
%%%%%%%%%%%%%%%%%%%%%%%%%%%%%%%%%%%%%

We start with some necessary definitions and notations used in the article. Set $\w:=\{ 0,1,2,\dots\}$.
All topological spaces are assumed to be Tychonoff (= completely regular and $T_1$). The closure of a subset $A$ of a topological space $X$ is denoted by $\overline{A}$, $\overline{A}^X$ or $\cl_X(A)$. A function $f:X\to Y$ between topological spaces $X$ and $Y$ is called {\em sequentially continuous} if for any convergent sequence $\{x_n\}_{n\in\w}\subseteq X$, the sequence $\{f(x_n)\}_{n\in\w}$ converges in $Y$ and $\lim_{n}f(x_n)=f(\lim_{n}x_n)$.
A subset $A$ of a topological space $X$ is called {\em functionally bounded in $X$} if every $f\in C(X)$ is bounded on $A$.
%\begin{enumerate}
%\item[$\bullet$] {\em $C$-embedded} in $X$ if every $f\in C(A)$ can be extended to an ${\bar f}\in C(X)$;
%\item[$\bullet$] {\em $C^\ast$-embedded}  in $X$ if every bounded $f\in C(A)$ can be extended to a bounded ${\bar f}\in C(X)$;
%\item[$\bullet$] {\em relatively compact} if its closure ${\bar A}$ is compact;
%\item[$\bullet$] {\em relatively countably compact} if each countably infinite subset in $A$ has a cluster point in $X$;
%\item[$\bullet$] ({\em relatively}) {\em sequentially compact} if each sequence in $A$ has a subsequence converging to a point of $A$ (resp., of $X$);
%\item[$\bullet$] {\em countably compact}  if each countably infinite subset in $A$ has a cluster point in $A$;
%\item[$\bullet$] {\em sequentially compact} if each sequence in $A$ has a subsequence converging to a point of $A$;
%\item[$\bullet$] {\em functionally bounded in $X$} if every $f\in C(X)$ is bounded on $A$.
%\end{enumerate}

All topological vector spaces are over the field $\ff$ of real or complex numbers.  The closed unit ball of the field $\ff$ is denoted by $\mathbb{D}$.

Let $E$ be a locally convex space. The span of a subset $A$ of $E$ and its closure are denoted by $E_A:=\spn(A)$ and $\cspn(A)$, respectively. We denote by $\Nn_0(E)$ (resp., $\Nn_{0}^c(E)$) the family of all (resp., closed absolutely convex) neighborhoods of zero of $E$. The family of all bounded subsets of $E$ is denoted by $\Bo(E)$. The value of $\chi\in E'$ on $x\in E$ is denoted by $\langle\chi,x\rangle$ or $\chi(x)$. A sequence $\{x_n\}_{n\in\w}$ in $E$ is said to be {\em Cauchy} if for every $U\in\Nn_0(E)$ there is $N\in\w$ such that $x_n-x_m\in U$ for all $n,m\geq N$. It is easy to see that a sequence $\{x_n\}_{n\in\w}$ in  $E$ is Cauchy if and only if $x_{n_k}-x_{n_{k+1}}\to 0$ for every (strictly) increasing sequence $(n_k)$ in $\w$. If $E$ is a normed space, $B_E$ denotes the closed unit ball of $E$.

For an lcs $E$, we denote by $E_w$ and $E_\beta$ the space $E$ endowed with the weak topology $\sigma(E,E')$ and with the strong topology $\beta(E,E')$, respectively. The topological dual space $E'$ of $E$ endowed with weak$^\ast$ topology $\sigma(E',E)$ or with the strong topology $\beta(E',E)$ is denoted by $E'_{w^\ast}$ or $E'_\beta$, respectively. The closure of a subset $A$ in the weak topology is denoted by $\overline{A}^{\,w}$ or $\overline{A}^{\,\sigma(E,E')}$, and $\overline{B}^{\,w^\ast}$ (or $\overline{B}^{\,\sigma(E',E)}$) denotes the closure of $B\subseteq E'$ in the weak$^\ast$ topology. The {\em polar} of a subset $A$ of $E$ is denoted by
$
A^\circ :=\{ \chi\in E': \|\chi\|_A \leq 1\}. %, \quad\mbox{ where }\quad\|\chi\|_A=\sup\big\{|\chi(x)|: x\in A\cup\{0\}\big\}.
$
A subset $B$ of $E'$ is {\em equicontinuous} if $B\subseteq U^\circ$ for some $U\in \Nn_0(E)$.

A subset $A$ of a locally convex space $E$ is called
\begin{enumerate}
\item[$\bullet$] {\em precompact} if for every $U\in\Nn_0(E)$ there is a finite set $F\subseteq E$ such that $A\subseteq F+U$;
\item[$\bullet$] {\em sequentially precompact} if every sequence in $A$ has a Cauchy subsequence;
\item[$\bullet$] {\em weakly $($sequentially$)$ compact} if $A$ is (sequentially) compact in $E_w$;
\item[$\bullet$] {\em relatively weakly compact} if its weak closure $\overline{A}^{\,\sigma(E,E')}$ is compact in $E_w$;
\item[$\bullet$] {\em relatively weakly sequentially compact} if each sequence in $A$ has a subsequence weakly converging to a point of $E$;
\item[$\bullet$] {\em weakly sequentially precompact} if each sequence in $A$ has a weakly Cauchy subsequence.
\end{enumerate}
Note that each sequentially precompact subset of $E$ is precompact, by the converse is not true in general, see Lemma 2.2 of \cite{Gab-Pel}. %\ref{l:seq-precom-precom}.

In what follows we shall actively use the following classical completeness type properties and weak barrelledness conditions. A locally convex space  $E$
\begin{enumerate}
%\item[$\bullet$] is {\em complete} if each Cauchy net converges;
\item[$\bullet$] is {\em quasi-complete} if each closed bounded subset of $E$ is complete;
%\item[$\bullet$] is {\em von Neumann complete} if every precompact subset of $E$ is relatively compact; %($N$-complete)
\item[$\bullet$] is {\em sequentially complete} if each Cauchy sequence in $E$ converges;
\item[$\bullet$] is {\em locally complete} if the closed absolutely convex hull of a null sequence in $E$ is compact;
%\item[$\bullet$] is {\em dual locally complete} (dlc) if $(E',\sigma(E',E))$ is locally complete;
%\item[$\bullet$] has the {\em convex compactness property} (ccp) if the closed absolutely convex hull of each compact subset $K$ of $E$ is compact;
%\item[$\bullet$] has the {\em Krein property} if the closed absolutely convex hull of each weakly compact subset $K$ of $E$ is weakly compact (i.e., $E_w$ has the ccp).
\item[$\bullet$] ({\em quasi}){\em barrelled} if every $\sigma(E',E)$-bounded (resp., $\beta(E',E)$-bounded) subset of $E'$ is equicontinuous;
\item[$\bullet$] {\em $c_0$-}({\em quasi}){\em barrelled} if every $\sigma(E',E)$-null (resp., $\beta(E',E)$-null) sequence is equicontinuous.
\end{enumerate}
It is well known that $C_p(X)$ is quasibarrelled for every Tychonoff space $X$.

Recall that a locally convex space $(E,\tau)$ has the {\em Schur property} (resp., the {\em Glicksberg  property}) if $E$ and $E_w$ have the same convergent sequences (resp., the same compact sets). If an lcs $E$ has the  Glicksberg property, then it has the Schur property. The converse is true for strict $(LF)$-spaces (in particular, for Banach spaces), but not in general, see Corollary 2.13 and Proposition 3.5 of \cite{Gabr-free-resp}.
We shall use the next two lemmas.
\begin{lemma} \label{l:Schur-rel-seq-comp}
A locally convex space $(E,\tau)$ has the Schur property if and only if $E$ and $E_w$ have the same relatively sequentially compact sets.
\end{lemma}

\begin{proof}
Assume that $(E,\tau)$ has the Schur property. If $A$ is a relatively sequentially compact subset of $E$, then evidently $A$ is relatively sequentially compact in $E_w$. Conversely, let $A$ be a relatively sequentially compact subset of $E_w$. Take a sequence $S=\{ a_n\}_{n\in\w}$ in $A$. Then $S$ has a subsequence $\{ a_{n_k}\}_{k\in\w}$ weakly converging to a point $x\in E$. By the Schur property $a_{n_k}\to x$ in $\tau$. Hence $A$ is a relatively sequentially compact subset of $E$. Thus $E$ and $E_w$ have the same relatively sequentially compact sets. %$E$ weakly respects sequential compactness.

Assume that $E$ and $E_w$ have the same relatively sequentially compact sets. To show that $E$ has the Schur property, let $S=\{ a_n\}_{n\in\w}$ be a weakly null sequence. Then $S$ is relatively sequentially compact in $E_w$ and hence also in $E$.  We show that $a_n\to 0$ also in $E$. Suppose for a contradiction that there is a $U\in\Nn_0(E)$ such that $a_n\not\in U$ for each $n\in I$ for some infinite $I\subseteq\w$. Since $S$ is relatively sequentially compact in $E$, the sequence $S'=\{a_n\}_{n\in I}$  has a subsequence $\{ b_n\}_{n\in\w}$ converging to some point $b\in E$. As $\{ b_n\}_{n\in\w}$ is also weakly null we have $b=0$. But since $\{ b_n\}_{n\in\w}\subseteq S'$ it follows that $b_n\not\in U$ for every $n\in\w$; so $b_n\not\to b$, a contradiction.\qed
\end{proof}

\begin{lemma} \label{l:seq-precom-Schur}
Every weakly sequentially precompact subset $A$ of a Schur space $E$ is  sequentially precompact. If in addition $E$ is sequentially complete, then $E_w$ is sequentially complete.
\end{lemma}

\begin{proof}
Let $\{x_n\}_{n\in\w}$ be a sequence in $A$. As $A$ is  weakly sequentially precompact, there is a subsequence  $\{y_n\}_{n\in\w}$ of  $\{x_n\}_{n\in\w}$ which is weakly Cauchy. Let $(n_k)$ be a strictly increasing sequence in $\w$. Then  $\{y_{n_{k+1}}-y_{n_{k}}\}_{n\in\w}$ is weakly null. By the Schur property of $E$, we obtain that $y_{n_{k+1}}-y_{n_{k}}\to 0$ in $E$. Thus $\{y_n\}_{n\in\w}$ is a Cauchy sequence in $E$, and hence $A$ is  sequentially precompact.

Assume that $E$ is in addition a sequentially complete space. Let $S=\{x_n\}_{n\in\w}$ be a weakly Cauchy sequence. As we proved above $S$ is a Cauchy sequence in $E$. Since $E$ is sequentially complete, there is $a\in E$ such that $x_n\to a$. Thus $x_n$ converges to $a$ also in the weak topology.\qed
\end{proof}

Two vector topologies $\tau$ and $\TTT$ on a vector space $L$ are called {\em compatible} if $(L,\tau)'=(L,\TTT)'$ algebraically. If $(E,\tau)$ is a locally convex space, then there is a finest locally convex vector topology $\mu(E,E')$  compatible with $\tau$. The topology $\mu(E,E')$ is called the {\em Mackey topology}, and if $\tau=\mu(E,E')$, the space $E$ is called a {\em Mackey space}. Set $E_\mu:=\big(E,\mu(E,E')\big)$. It is well known that any quasibarrelled space is Mackey.

Recall that an lcs $E$ is called {\em semi-reflexive} if the canonical map $J_E:E\to E''=(E'_\beta)'_\beta$ defined by  $\langle J_E(x),\chi\rangle:=\langle\chi,x\rangle$ ($\chi\in E'$) is an isomorphism; if in addition $J_E$ is a topological isomorphism, the space $E$ is called {\em reflexive}. Each reflexive space is barrelled.

We denote by $\bigoplus_{i\in I} E_i$ and $\prod_{i\in I} E_i$  the locally convex direct sum and the topological product of a non-empty family $\{E_i\}_{i\in I}$ of locally convex spaces, respectively. If $0\not= \xxx=(x_i)\in \bigoplus_{i\in I} E_i$, then the set $\supp(\xxx):=\{i\in I: x_i\not= 0\}$ is called the {\em support} of $\xxx$. The {\em support}  of a subset $A$, $\{0\}\subsetneq A$, of $\bigoplus_{i\in I} E_i$ is the set $\supp(A):=\bigcup_{a\in A} \supp(a)$. We shall also consider elements $\xxx=(x_i) \in \prod_{i\in I} E_i$ as functions on $I$ and write $\xxx(i):=x_i$.

%We denote by $\ind_{n\in \w} E_n$  the inductive limit of a (reduced) inductive sequence $\big\{(E_n,\tau_n)\big\}_{n\in \w}$  of locally convex spaces. If in addition $\tau_m{\restriction}_{E_n} =\tau_n$ for all $n,m\in\w$ with $n\leq m$, the inductive limit $\ind_{n\in \w} E_n$ is called {\em strict} and is denoted by $\SI E_n$. It is well known that $E:=\SI E_n$ is regular, i.e., every bounded set in $E$ is contained in some $E_n$. For more details we refer the reader to Section 4.5 of \cite{Jar}. In the partial case when all spaces $E_n$ are Fr\'{e}chet, the strict inductive limit is called a {\em strict $(LF)$-space}. One of the most important examples of strict $(LF)$-spaces is  the space $\mathcal{D}(\Omega)$ of test functions over an open subset $\Omega$ of $\IR^n$. The strong dual $\mathcal{D}'(\Omega)$ of $\mathcal{D}(\Omega)$ is the space of distributions.

Below we recall some of the basic classes of compact-type operators.

\begin{definition} \label{def:operators}
Let $E$ and $L$ be locally convex spaces. An operator $T\in \LL(E,L)$ is called {\em compact} (resp., {\em sequentially compact, precompact, sequentially precompact, weakly compact, weakly sequentially compact, weakly sequentially precompact, bounded}) if there is $U\in \Nn_0(E)$ such that $T(U)$ a relatively compact (relatively sequentially compact,  precompact, sequentially precompact,  relatively weakly compact, relatively weakly sequentially compact,  weakly sequentially precompact or bounded) subset of $E$. \qed
\end{definition}

Let $p\in[1,\infty]$. Then $p^\ast$ is defined to be the unique element of $ [1,\infty]$ which satisfies $\tfrac{1}{p}+\tfrac{1}{p^\ast}=1$. For $p\in[1,\infty)$, the space $\ell_{p^\ast}$  is the dual space of $\ell_p$. We denote by $\{e_n\}_{n\in\w}$ the canonical basis of $\ell_p$, if $1\leq p<\infty$, or the canonical basis of $c_0$, if $p=\infty$. The canonical basis of $\ell_{p^\ast}$ is denoted by $\{e_n^\ast\}_{n\in\w}$. %In what follows we usually identify $\ell_{1^\ast}$ with $c_0$.
Denote by  $\ell_p^0$ and $c_0^0$ the linear span of $\{e_n\}_{n\in\w}$  in  $\ell_p$ or $c_0$ endowed with the induced norm topology, respectively.
We shall use repeatedly the following well known description of relatively compact subsets of $\ell_p$ and $c_0$,  see \cite[p.~6]{Diestel}.
\begin{proposition} \label{p:compact-ell-p}
{\rm(i)} A bounded subset $A$ of $\ell_p$, $p\in[1,\infty)$, is relatively compact if and only if
\[
\lim_{m\to\infty} \sup\Big\{ \sum_{m\leq n} |x_n|^p : x=(x_n)\in A\Big\} =0.
\]
{\rm(ii)} A bounded subset $A$ of $c_0$ is relatively compact if and only if
$
\lim_{n\to\infty} \sup\{ |x_n|: x=(x_n)\in A\} =0.
$
\end{proposition}

One of the most important classes of locally convex spaces is the class of free locally convex spaces introduced by Markov in \cite{Mar}. The {\em  free locally convex space}  $L(X)$ over a Tychonoff space $X$ is a pair consisting of a locally convex space $L(X)$ and  a continuous map $i: X\to L(X)$  such that every  continuous map $f$ from $X$ to a locally convex space  $E$ gives rise to a unique continuous linear operator $\Psi_E(f): L(X) \to E$  with $f=\Psi_E(f) \circ i$. The free locally convex space $L(X)$ always exists and is essentially unique, and $X$ is the Hamel basis of $L(X)$. So, each nonzero $\chi\in L(X)$ has a unique decomposition $\chi =a_1 i(x_1) +\cdots +a_n i(x_n)$, where all $a_k$ are nonzero and $x_k$ are distinct. The set $\supp(\chi):=\{x_1,\dots,x_n\}$ is called the {\em support} of $\chi$. In what follows we shall identify $i(x)$ with $x$ and consider $i(x)$ as the Dirac measure $\delta_x$ at the point $x\in X$. We also recall that $C_p(X)'=L(X)$ and $L(X)'=C(X)$. It is worth mentioning that $L(X)$ has the Glicksberg property for every Tychonoff space $X$, and if $X$ is non-discrete, then $L(X)$ is not a Mackey space, see \cite{Gab-Respected} and \cite{Gabr-L(X)-Mackey}, respectively.

%The following notions will play a crucial role in the article.
%\begin{definition}\label{def:tvs-weakly-p-sum}{\em
Let $p\in[1,\infty]$. A sequence  $\{x_n\}_{n\in\w}$ in a locally convex space $E$ is called
\begin{enumerate}
\item[$\bullet$]  {\em weakly $p$-summable} if  for every $\chi\in E'$, it follows
\[
\mbox{$(\langle\chi, x_n\rangle)_{n\in\w} \in\ell_p$ if $p<\infty$, and  $(\langle\chi, x_n\rangle)_{n\in\w} \in c_0$ if $p=\infty$;}
\]
\item[$\bullet$] {\em weakly $p$-convergent to $x\in E$} if  $\{x_n-x\}_{n\in\w}$ is weakly $p$-summable;
\item[$\bullet$] {\em weakly $p$-Cauchy} if for each pair of strictly increasing sequences $(k_n),(j_n)\subseteq \w$, the sequence  $(x_{k_n}-x_{j_n})_{n\in\w}$ is weakly $p$-summable.
\end{enumerate}%}
%\end{definition}
%We shall denote by $\ell^w_p(E)$ the family of all weakly $p$-summable sequences in $E$. If $p=\infty$, for the simplicity of notations and formulations we shall usually identify $\ell^w_\infty(E)$ with the space $c_0^w(E)$ of all weakly null-sequences in $E$. Elements of $\ell^w_p(E)$ will be written as $(x_n)$.
%
%Let $p\in[1,\infty]$, and let $E$ be a locally convex space.
A sequence  $\{\chi_n\}_{n\in\w}$ in $E'$ is called  {\em weak$^\ast$ $p$-summable} (resp., {\em weak$^\ast$ $p$-convergent to $\chi\in E'$} or  {\em weak$^\ast$ $p$-Cauchy})  if it is weakly $p$-summable (resp., weakly $p$-convergent to $\chi\in E'$ or weakly $p$-Cauchy) in $E'_{w^\ast}$.
%A sequence  $\{\chi_n\}_{n\in\w}$ in $E'$ is called
%\begin{enumerate}
%\item[$\bullet$] {\em weak$^\ast$ $p$-summable} if it is weakly $p$-summable in $E'_{w^\ast}$, i.e., if  for every $x\in E$, it follows that $(\langle\chi_n, x\rangle)_{n\in\w} \in\ell_p$ if $p<\infty$, or  $(\langle\chi_n, x\rangle)_{n\in\w} \in c_0$ if $p=\infty$;
%\item[$\bullet$] {\em weak$^\ast$ $p$-convergent to $\chi\in E'$} if  $\{\chi_n-\chi\}_{n\in\w}$ is  weakly $p$-summable in $E'_{w^\ast}$;
%\item[$\bullet$] {\em weak$^\ast$ $p$-Cauchy} if it is weakly $p$-Cauchy in $E'_{w^\ast}$, i.e.,  if for each pair of strictly increasing sequences $(k_n),(j_n)\subseteq \w$, the sequence  $\{\chi_{k_n}-\chi_{j_n}\}_{n\in\w}$ is weak$^\ast$ $p$-summable.
%\end{enumerate}

The following weak barrelledness conditions introduced and studied in \cite{Gab-Pel} will play a considerable role in the article.
Let $p\in[1,\infty]$. A locally convex space $E$ is called
\begin{enumerate}
\item[$\bullet$]  {\em $p$-barrelled } if every weakly $p$-summable sequence in $E'_{w^\ast}$ is equicontinuous;
\item[$\bullet$] {\em $p$-quasibarrelled } if every weakly $p$-summable sequence in $E'_\beta$ is equicontinuous.
\end{enumerate}

We shall consider also the following linear map introduced in \cite{Gab-Pel}
\[
S_p:\LL(E,\ell_p)\to \ell_p^w(E'_{w^\ast}) \quad \big(\mbox{or $S_\infty: \LL(E,c_0)\to c_0^w(E'_{w^\ast})$ if $p=\infty$}\big)
\]
defined by $S_p(T):=\big(T^\ast(e^\ast_n)\big)_{n\in\w}$.

The following class of subsets of an lcs $E$ was introduced and studied in \cite{Gab-Pel}, and it generalizes the notion of $p$-$(V^\ast)$ subsets  of Banach spaces defined in \cite{CCDL}.
\begin{definition}\label{def:tvs-V*-subset}{\em
Let $p,q\in[1,\infty]$. A non-empty subset   $A$ of a locally convex space $E$ is called a  {\em $(p,q)$-$(V^\ast)$ set} (resp., a {\em $(p,q)$-$(EV^\ast)$ set}) if
\[
\Big(\sup_{a\in A} |\langle \chi_n, a\rangle|\Big)\in \ell_q \; \mbox{ if $q<\infty$, } \; \mbox{ or }\;\; \Big(\sup_{a\in A} |\langle \chi_n, a\rangle|\Big)\in c_0 \; \mbox{ if $q=\infty$},
\]
for every  (resp., equicontinuous) weakly $p$-summable sequence $\{\chi_n\}_{n\in\w}$ in  $E'_\beta$. $(p,\infty)$-$(V^\ast)$ sets and $(1,\infty)$-$(V^\ast)$ sets will be called simply {\em $p$-$(V^\ast)$ sets} and {\em $(V^\ast)$ sets}, respectively. Analogously, $(p,\infty)$-$(EV^\ast)$ sets and $(1,\infty)$-$(EV^\ast)$ sets will be called {\em $p$-$(EV^\ast)$ sets} and {\em $(EV^\ast)$ sets}, respectively. \qed}
\end{definition}
The family of all $(p,q)$-$(V^\ast)$ sets (resp. $p$-$(V^\ast)$ sets, $(p,q)$-$(EV^\ast)$ sets, $(V^\ast)$ sets etc.) of an lcs $E$ is denoted by $\mathsf{V}^\ast_{(p,q)}(E)$ (resp. $\mathsf{V}^\ast_{p}(E)$, $\mathsf{EV}^\ast_{(p,q)}(E)$, $\mathsf{V}^\ast(E)$ etc.).

Following \cite{Gab-Pel}, a non-empty subset $B$ of $E'$ is called a  {\em $(p,q)$-$(V)$ set} if
\[
\Big(\sup_{\chi\in B} |\langle \chi, x_n\rangle|\Big)\in \ell_q \; \mbox{ if $q<\infty$, } \; \mbox{ or }\;\; \Big(\sup_{\chi\in B} |\langle \chi, x_n\rangle|\Big)\in c_0 \; \mbox{ if $q=\infty$},
\]
for every weakly $p$-summable sequence $\{x_n\}_{n\in\w}$ in  $E$. $(p,\infty)$-$(V)$ sets and $(1,\infty)$-$(V)$ sets  will be called simply {\em $p$-$(V)$ sets} and {\em $(V)$ sets}, respectively.

Let $1\leq p\leq q\leq \infty$, and let $E$ and $L$ be locally convex spaces. Following \cite{Gab-Pel}, a linear map $T:E\to L$ is called {\em $(q,p)$-convergent} if it sends weakly $p$-summable sequences in $E$ to strongly $q$-summable sequences in $L$.

The following $p$-versions of weakly compact-type properties are defined in \cite{Gab-Pel} generalizing the corresponding notions in the class of Banach spaces introduced in \cite{CS} and \cite{Ghenciu-pGP}. Let $p\in[1,\infty]$. A subset   $A$ of a locally convex space $E$ is called
\begin{enumerate}
\item[$\bullet$] ({\em relatively}) {\em weakly sequentially $p$-compact} if every sequence in $A$ has a weakly $p$-convergent  subsequence with limit in $A$ (resp., in $E$);
\item[$\bullet$] {\em weakly  sequentially $p$-precompact} if every sequence from $A$ has a  weakly $p$-Cauchy subsequence.
\end{enumerate}

A Tychonoff space $X$  is called {\em Fr\'{e}chet--Urysohn} if for any cluster point $a\in X$ of a subset $A\subseteq X$ there is a sequence $\{ a_n\}_{n\in\w}\subseteq A$ which converges to $a$. A Tychonoff space $X$ is called an {\em angelic space} if (1) every relatively countably compact subset of $X$ is relatively compact, and (2) any compact subspace of $X$ is Fr\'{e}chet--Urysohn. Note that any subspace of an angelic space is angelic, and a subset $A$ of an angelic space $X$ is compact if and only if it is countably compact if and only if $A$ is sequentially compact, see Lemma 0.3 of \cite{Pryce}.
%Being motivated by the last property we introduced in \cite{Gabr-free-resp} a new class of  Tychonoff spaces:  $X$ is called {\em sequentially angelic} if a subset $K$ of $X$ is compact if and only if $K$ is sequentially compact. The class of sequentially angelic spaces is strictly wider than the class of angelic spaces (indeed, if $X$ is a countably compact space without infinite compact subsets, then $X$ is trivially sequentially angelic but not angelic).

Let $p\in[1,\infty]$. Following \cite{Gab-DP}, a locally convex space $E$ is called a {\em weakly sequentially $p$-angelic space} if the family of all relatively weakly sequentially $p$-compact sets in $E$ coincides with the family of all relatively weakly compact subsets of $E$. The space $E$ is a {\em weakly $p$-angelic space}  if it is  a weakly sequentially $p$-angelic space and each weakly compact subset  of $E$ is Fr\'{e}chet--Urysohn.

%%%%%%%%%%%%%%%%%%%%%%%%%%%%%%%%%%%%%%%%%%%
%%%%%%%%%%%%%%%%%%%%%%%%%%%%%%%%%%%%%%%%%%%
%%%%%%%%%%%%%%%%%%%%%%%%%%%%%%%%%%%%%%%%%%%
%%%%%%%%%%%%%%%%%%%%%%%%%%%%%%%%%%%%%%%%%%%
%%%%%%%%%%%%%%%%%%%%%%%%%%%%%%%%%%%%%%%%%%%

\section{Limited-type sets in locally convex spaces } \label{sec:limit-set}

%%%%%%%%%%%%%%%%%%%%%%%%%%%%%%%%%%%%%%%%%%%
%%%%%%%%%%%%%%%%%%%%%%%%%%%%%%%%%%%%%%%%%%%
%%%%%%%%%%%%%%%%%%%%%%%%%%%%%%%%%%%%%%%%%%%
%%%%%%%%%%%%%%%%%%%%%%%%%%%%%%%%%%%%%%%%%%%
%%%%%%%%%%%%%%%%%%%%%%%%%%%%%%%%%%%%%%%%%%%

In the next lemma we summarize some basic elementary properties of $(p,q)$-limited sets, cf. (i) of Theorem \ref{t:limited-BD}. %can be proved repeating word for word the proof of Lemma \ref{l:V*-set-1}.
\begin{lemma} \label{l:limited-set-1}
Let $p,q\in[1,\infty]$, and let $(E,\tau)$ be a locally convex space. Then:
\begin{enumerate}
\item[{\rm(i)}] every $(p,q)$-limited set is $(p,q)$-$\EE$-limited; the converse is true if $E$ is a $p$-barrelled space;
\item[{\rm(ii)}] every $(p,q)$-$\EE$-limited set in $E$ is bounded;
\item[{\rm(iii)}] the family of all $(p,q)$-limited $($resp., $(p,q)$-$\EE$-limited$)$ sets in $E$ is closed under taking subsets, finite unions and sums, and closed absolutely convex hulls;
\item[{\rm(iv)}] the family of all $(p,q)$-limited $($resp., $(p,q)$-$\EE$-limited$)$ sets in $E$ is closed under taking continuous linear images; in particular,  if $H$ is a subspace of $E$, then every $(p,q)$-limited $($resp., $(p,q)$-$\EE$-limited$)$ set in $H$ is $(p,q)$-limited $($resp., $(p,q)$-$\EE$-limited$)$ in $E$;
\item[{\rm(v)}] a  subset $A$ of $E$ is a $(p,q)$-limited $($resp., $(p,q)$-$\EE$-limited$)$ set  if and only if every countable subset of $A$ is a  $(p,q)$-limited $($resp., $(p,q)$-$\EE$-limited$)$ set;
\item[{\rm(vi)}] if $p',q'\in[1,\infty]$ are such that $p'\leq p$ and $q\leq q'$, then every $(p,q)$-limited $($resp.,  $(p,q)$-$\EE$-limited$)$ set in $E$ is also $(p',q')$-limited $($resp.,  $(p',q')$-$\EE$-limited$)$; in particular, any $(p,q)$-limited $($resp., $(p,q)$-$\EE$-limited$)$ set is $(1,\infty)$-limited $($resp., $(1,\infty)$-$\EE$-limited$)$;
\item[{\rm(vii)}] the property of being a $(p,q)$-limited set depends only on the duality $(E,E')$, i.e.,   if $\TTT$ is a locally convex vector topology on $E$ compatible with the topology $\tau$ of $E$, then the  $(p,q)$-limited sets of $(E,\TTT)$ are exactly the $(p,q)$-limited sets of $(E,\tau)$;
\item[{\rm(viii)}] every $(p,q)$-limited $($resp., $(p,q)$-$\EE$-limited$)$ set in $E$ is a $(p,q)$-$(V^\ast)$ $($resp., $(p,q)$-$(EV^\ast)$$)$ set; the converse is true for semi-reflexive spaces;
\item[{\rm(ix)}] every $(p,q)$-limited subset of $E'_\beta$ is a $(p,q)$-$(V)$ set;
\item[{\rm(x)}] if $q<p$ and $A$ is a $(p,q)$-$\EE$-limited subset of $E$, then $A=\{0\}$;
\item[{\rm(xi)}] if $q\geq p$, then any finite subset of $E$ is $(p,q)$-limited;
\item[{\rm(xii)}]  a bounded subset $A$ of $E$ is $(p,q)$-limited $($resp., $(p,q)$-$\EE$-limited$)$ if and only if for every sequence $\{x_n\}_{n\in\w}$ in $A$ and each $($resp., equicontinuous$)$ weak$^\ast$ $p$-summable sequence $\{\chi_n\}_{n\in\w}$ in $E'$, it follows $\big(|\langle \chi_n, x_n\rangle|\big)\in \ell_q$ $($or $\in c_0$ if $p=\infty$$)$.
\end{enumerate}
\end{lemma}

\begin{proof}
(i) and (iii) are clear, and (viii) follows from Definitions \ref{def:p-limit-coarse-set} and \ref{def:tvs-V*-subset}  and the trivial fact that every  (equicontinuous) weakly $p$-summable sequence $\{\chi_n\}_{n\in\w}$ in  $E'_\beta$ is (resp., equicontinuous)  weak$^\ast$ $p$-summable in  $E'$. The clause (ii) follows from (viii) and (ii) of Lemma 7.2 of \cite{Gab-Pel} (which states that every $(p,q)$-$(EV^\ast)$ set is bounded). %\ref{l:V*-set-1}.
(vii) follows from  the definition of $(p,q)$-limited sets, and (ix) follows from the easy fact that for every weakly $p$-summable sequence $\{x_n\}_{n\in\w}$ in  $E$, the sequence  $\{J_E(x_n)\}_{n\in\w}$ is weak$^\ast$ $p$-summable in $E''$.
\smallskip

%(ii) Let $A$ be a $(p,q)$-$\EE$-limited set and suppose for a contradiction that it is unbounded. Then there is $\chi\in E'$ such that for every $n\in\w$, there exists $a_n\in A$ for which $|\langle \chi, a_n\rangle|> (n+1)^{3}$. For every $n\in\w$, set $\chi_n:=\tfrac{1}{(n+1)^{2}}\cdot \chi$. Then $\{\chi_n\}_{n\in\w}$ is an equicontinuous weakly $p$-summable sequence in  $E'_\beta$. However, since
%\[
%\sup_{a\in A} |\langle \chi_n, a\rangle| \geq |\langle \chi_n, a_n\rangle|> n+1 \to \infty
%\]
%it follows that $A$ is not a $(p,q)$-$\EE$-limited set, a contradiction.
%\smallskip

(iv) Let $T:E\to L$ be an operator from $E$ to an lcs $L$, and let $A$ be a $(p,q)$-limited (resp., $(p,q)$-$\EE$-limited)  set in $E$. Observe that the adjoint map $T^\ast: L'_{w^\ast}\to E'_{w^\ast}$ is continuous. Fix a (resp., equicontinuous) weak$^\ast$ $p$-summable sequence $S=\{\chi_n\}_{n\in\w}$ in  $L'$. It is easily seen (see Lemma 4.5 of \cite{Gab-Pel}) %\ref{l:prop-p-sum}(iii),
that the sequence $\{T^\ast(\chi_n)\}$ is weak$^\ast$  $p$-summable in  $E'$. If in addition the sequence $S$  is equicontinuous, then its image $T^\ast(S)$ is equicontinuous as well.
%(take $V\in\Nn_0(L)$ such that $S\subseteq V^\circ$ and put $U:=T^{-1}(V)\in\Nn_0(E)$. Then for every $x\in U$ and $n\in\w$ we obtain $ |\langle T^\ast(\chi_n),x\rangle|=|\langle \chi_n,T(x)\rangle|\leq 1$ and hence $T^\ast(\chi_n)\in U^\circ$; which means that the sequence  $(T^\ast(\chi_n))$  is equicontinuous.
Therefore
\[
\Big(\sup_{a\in A} |\langle \chi_n, T(a)\rangle|\Big)= \Big(\sup_{a\in A} |\langle T^\ast (\chi_n), a\rangle|\Big)\in \ell_q \; \mbox{ (or } \;  \in c_0 \; \mbox{ if $q=\infty$}).
\]
Thus $T(A)$ is a $(p,q)$-limited (resp., $(p,q)$-$\EE$-limited)  set in $L$.

The last assertion follows from the proved one applied to the identity embedding $T:H\to E$.
\smallskip

(v) The necessity follows from (iii). To prove the sufficiency suppose for a contradiction that $A$ is not a $(p,q)$-limited (resp., $(p,q)$-$\EE$-limited)  set in $E$. Then there is a (resp., equicontinuous)  weak$^\ast$ $p$-summable sequence $\{\chi_n\}_{n\in\w}$ in  $E'$ such that
\[
\Big(\sup_{a\in A} |\langle \chi_n, a\rangle|\Big)\not\in \ell_q \; \mbox{ if $q<\infty$, } \; \mbox{ or }\;\; \Big(\sup_{a\in A} |\langle \chi_n, a\rangle|\Big)\not\in c_0 \; \mbox{ if $q=\infty$}.
\]
Assume that $q<\infty$ (the case $q=\infty$ can be considered analogously). For every $n\in\w$, choose $a_n\in A$ such that $|\langle \chi_n, a_n\rangle|\geq\tfrac{1}{2} \cdot \sup_{a\in A} |\langle \chi_n, a\rangle|$. Then
\[
\sum_{n\in\w} |\langle \chi_n, a_n\rangle|^q \geq \tfrac{1}{2^q} \sum_{n\in\w} \big(\sup_{a\in A} |\langle \chi_n, a\rangle|\big)^q =\infty.
\]
Thus the countable subset $\{a_n\}_{n\in\w}$ of $A$ is not a $(p,q)$-limited (resp., $(p,q)$-$\EE$-limited)  set in $E$, a contradiction.
\smallskip

(vi) Take any  (resp., equicontinuous)  weak$^\ast$ $p'$-summable sequence $\{\chi_n\}_{n\in\w}$ in  $E'$. Since $p'\leq p$,  $\{\chi_n\}_{n\in\w}$  is also (resp., equicontinuous)  weak$^\ast$  $p$-summable and hence $\big( \sup_{a\in A} |\langle \chi_n, a\rangle|\big)\in \ell_q$ (or $\in c_0$ if $q=\infty$). It remains to note that $\ell_q\subseteq \ell_{q'}$ because $q\leq q'$.

(x) Let $q<p$ and $A$ be a $(p,q)$-$\EE$-limited subset of $E$. Then, by (viii), $A$ is a $(p,q)$-$(EV^\ast)$ set. Therefore, by Proposition 7.5 of \cite{Gab-Pel}, $A=\{0\}$.
%
%(ix) Let $A$ be a $(p,q)$-limited (resp., $(p,q)$-$\EE$-limited) set in $H$, and let $(\chi_n)$  be a  (resp., equicontinuous) weak$^\ast$ $p$-summable sequence in $E'$. Since $E'=H'$ algebraically, we obtain that  $(\chi_n)$ is  (resp., equicontinuous) weak$^\ast$ $p$-summable also in $H'$. Therefore,
%\[
%\Big(\sup_{a\in A} |\langle \chi_n, a\rangle|\Big)\in \ell_q \; \; \mbox{ (or $\in c_0$ if $q=\infty$)},
%\]
%which means that $A$ is a $(p,q)$-limited (resp., $(p,q)$-$\EE$-limited) set in $E$.\qed

(xi) By (iii) it suffices to show that $A=\{x\}$ is a $(p,q)$-limited set for every $x\in E$. Let $\{\chi_n\}_{n\in\w}$ be a weak$^\ast$ $p$-summable sequence in $E'_\beta$. Then $\big(\langle\chi_n,x\rangle\big)\in \ell_p$ (or $\in c_0$ if $p=\infty$). Since $p\leq q$ it follows that $\big(\sup_{x\in A}|\langle\chi_n,x\rangle|\big)\in \ell_q$ (or $\in c_0$ if $q=\infty$). Thus $A$ is a $(p,q)$-limited set.
\smallskip

(xii) The necessity is clear. To prove the sufficiency, for every $n\in\w$, choose $x_n\in A$ such that $|\langle \chi_n, x_n\rangle|\geq \tfrac{1}{2}\sup_{a\in A} |\langle \chi_n, a\rangle|$. By assumption, $\big(|\langle \chi_n, x_n\rangle|\big)\in \ell_q$ (or $\in c_0$ if $q=\infty$). Therefore also $\big(\sup_{a\in A} |\langle \chi_n, a\rangle|\big)\in \ell_q$  (or $\in c_0$ if $q=\infty$). Thus $A$ is a $(p,q)$-limited (resp., $(p,q)$-$\EE$-limited) set.\qed
\end{proof}
It follows from (x) and (xi) that there is sense to consider only the case when $1\leq p\leq q\leq\infty$.

\begin{notation} \label{n:limited-set} {\em
The family of all $(p,q)$-limited (resp., $p$-limited, limited, $(p,q)$-$\EE$-limited, $p$-$\EE$-limited, or $\EE$-limited) sets of an lcs $E$ is denoted by $\mathsf{L}_{(p,q)}(E)$ (resp., $\mathsf{L}_{p}(E)$, $\mathsf{L}(E)$, $\mathsf{EL}_{(p,q)}(E)$, $\mathsf{EL}_{p}(E)$, or $\mathsf{EL}(E)$).\qed}
\end{notation}

Below we characterize $(p,q)$-limited sets in products and direct sums.
\begin{proposition} \label{p:product-sum-limited-set}
Let  $1\leq p\leq q\leq\infty$, and let $\{E_i\}_{i\in I}$  be a non-empty family of locally convex spaces. Then:
\begin{enumerate}
\item[{\rm(i)}] a subset $K$ of $E=\prod_{i\in I} E_i$ is a  $(p,q)$-limited $($resp., $(p,q)$-$\EE$-limited$)$ set if and only if so are all its coordinate projections;
\item[{\rm(ii)}]  a subset $K$ of  $E=\bigoplus_{i\in I} E_i$  is a  $(p,q)$-limited  $($resp., $(p,q)$-$\EE$-limited$)$ set if and only if so are all its coordinate projections   and the support of $K$ is finite.
\end{enumerate}
\end{proposition}

\begin{proof}
The necessity follows from (iv) of Lemma \ref{l:limited-set-1} because $E_i$ is a direct summand of $E$ and, for the case (ii), the well known fact that any bounded subset of a locally convex direct sum has finite support.
\smallskip

To prove the sufficiency, let $K$ be a subset of $E$ such that each projection $K_i$ of $K$ is a $(p,q)$-limited (resp., $(p,q)$-$\EE$-limited) set in $E_i$, and, for the case (ii), $K_i=\{0\}$ for all but finitely many indices $i\in I$. We distinguish between the cases (i) and (ii).
\smallskip

(i) Take an arbitrary (resp., equicontinuous) weak$^\ast$ $p$-summable sequence $\{\chi_n\}_{n\in\w}$ in $E'$, where $\chi_n=(\chi_{i,n})_{i\in I}$. By Lemma 4.18 of \cite{Gab-Pel}, %\ref{l:support-p-sum},
the sequence  $\{\chi_n\}_{n\in\w}$ has finite support $F\subseteq I$ (i.e., $\chi_{i,n}=0$ for all $n\in\w$ and $i\in I\SM F$) and for every $i\in F$, each sequence $\{\chi_{i,n}\}_{n\in\w}$ is weak$^\ast$ $p$-summable in $E'_i$. If in addition $\{\chi_n\}_{n\in\w}$ is equicontinuous, then for every $i\in F$, the sequence $\{\chi_{i,n}\}_{n\in\w}\subseteq E'_i$ is also equicontinuous (indeed, if $T_i:E_i\to E$ is the identity embedding, then $\{\chi_{i,n}\}_n= \{ T^\ast_i(\chi_n)\}_n$ is equicontinuous). Then
\[
\sup_{x\in K} |\langle \chi_n, x\rangle|=\sup_{x\in K} \big|\sum_{i\in F} \langle \chi_{i,n}, x(i)\rangle\big|\leq \sum_{i\in F} \sup_{x(i)\in K_i} |\langle \chi_{i,n}, x(i)\rangle|.
\]
Since all $K_i$ are $(p,q)$-limited (resp., $(p,q)$-$\EE$-limited) sets, we have $\big( \sup_{x(i)\in K_i} |\langle \chi_{i,n}, x(i)\rangle|\big)  \in \ell_q$ (or $\in c_0$ if $q=\infty$). Therefore also  $\big( \sup_{x\in K} |\langle \chi_n, x\rangle|\big) \in \ell_q$ (or $\in c_0$ if $q=\infty$). Thus $K$ is  a $(p,q)$-limited (resp., $(p,q)$-$\EE$-limited)  set in $E$.
\smallskip

(ii) Let $F\subseteq I$ be the finite support of $K$. Take an arbitrary (resp., equicontinuous)  weak$^\ast$ $p$-summable sequence $\{\chi_n\}_{n\in\w}$ in $E'_\beta$, where $\chi_n=(\chi_{i,n})_{i\in I}$ with $\chi_{i,n}\in E'_i$. As in (i) above, if  $\{\chi_n\}_{n\in\w}$ is equicontinuous, then for every $i\in F$, the sequence $\{\chi_{i,n}\}_{n\in\w}\subseteq E'_i$ is also equicontinuous. Then, by Lemma 4.18 of \cite{Gab-Pel}, %\ref{l:support-p-sum},
for every $i\in F$, the sequence $\{\chi_{i,n}\}_{n\in\w}$ is weak$^\ast$ $p$-summable in $E'_i$ and hence
\[
\sup_{x\in K} |\langle \chi_n, x\rangle|=\sup_{x\in K} \big|\sum_{i\in F} \langle \chi_{i,n}, x(i)\rangle\big|\leq \sum_{i\in F} \sup_{x(i)\in K_i} |\langle \chi_{i,n}, x(i)\rangle|.
\]
Since all $K_i$ are $(p,q)$-limited  (resp., $(p,q)$-$\EE$-limited) sets, we have $\big( \sup_{x(i)\in K_i} |\langle \chi_{i,n}, x(i)\rangle|\big)  \in \ell_q$ (or $\in c_0$ if $q=\infty$). Therefore also  $\big( \sup_{x\in K} |\langle \chi_n, x\rangle|\big) \in \ell_q$ (or $\in c_0$ if $q=\infty$). Thus $K$ is  a $(p,q)$-limited  (resp., $(p,q)$-$\EE$-limited) set in $E'$.\qed
\end{proof}

%The following class of linear maps is defined and studied in Section 16 of \cite{Gab-Pel}.

%\begin{definition} \label{def:qp-summable} {\em
%Let $1\leq p\leq q\leq \infty$, and let $E$ and $L$ be locally convex spaces. A linear map $T:E\to L$ is called {\em $(q,p)$-convergent} if it sends weakly $p$-summable sequences in $E$ to strongly $q$-summable sequences in $L$. \qed}
%\end{definition}

Let $A$ be a bounded subset of a locally convex space $E$. Then, by Proposition 16.10 of \cite{Gab-Pel}, %\ref{p:L1-0-E},
 the map $T_A:\ell_1^0(A) \to E$ defined  by
\begin{equation} \label{equ:L1-E}
T_A(\lambda_0 a_0+\cdots+\lambda_n a_n):=\lambda_0 a_0+\cdots+\lambda_n a_n \quad (n\in\w, \; \lambda_0,\dots,\lambda_n\in\IF, \; a_0,\dots,a_n\in A),
\end{equation}
is an operator. Now we characterize $(p,q)$-limited sets. %, cf. Theorem 16.11 of \cite{Gab-Pel}.

\begin{theorem} \label{t:pq-limited-summing}
Let $1\leq p\leq q\leq\infty$, and let $E$ be a  {\rm(}resp., $p$-barrelled{\rm)} locally convex space. Then a bounded subset $A$ of  $E$ is a $(p,q)$-limited {\rm(}resp., $(p,q)$-$\EE$-limited{\rm)} set if and only if the adjoint operator $T_A^\ast: E'_{w^\ast}\to \ell_\infty(A)$ is $(q,p)$-convergent.
\end{theorem}

\begin{proof}
%By Proposition 16.10 of \cite{Gab-Pel}, %\ref{p:L1-0-E},
% every bounded subset $A$ of $E$ defines
Consider an operator $T_A:\ell_1^0(A) \to E$ defined in (\ref{equ:L1-E}).
Observe that for each $\chi\in E'$, the $a$th coordinate $T_A^\ast(\chi)(a)$ of $T_A^\ast(\chi)$ is
\[
T_A^\ast(\chi)(a) = \langle T_A^\ast(\chi),a\rangle=\langle\chi,T_A(a)\rangle=\langle\chi,a\rangle,
\]
and hence
\begin{equation} \label{equ:pq-limited-summing-1}
\|T^\ast_A(\chi)\|_{\ell_\infty(A)}=\sup_{a\in A} |T^\ast_A(\chi)(a)|=\sup_{a\in A}|\langle\chi,a\rangle|
\end{equation}

Now, by definition, a subset $A$ of $E$ is a  $(p,q)$-limited set if and only if $\big(\sup_{a\in A}|\langle\chi_n,a\rangle|\big)\in \ell_q$ (or $\in c_0$ if $q=\infty$) for every weak$^\ast$ $p$-summable sequence $\{\chi_n\}_{n\in\w}$ in $E'$, and hence, by  (\ref{equ:pq-limited-summing-1}), if and only if $\big(\|T^\ast_A(\chi_n)\|_{\ell_\infty(A)}\big)\in \ell_q$ (or $\in c_0$ if $q=\infty$) for every weakly $p$-summable sequence $\{\chi_n\}_{n\in\w}$ in $E'_{w^\ast}$, i.e., $T^\ast_A$  is a $(q,p)$-convergent linear map.

The case when  $E$ is $p$-barrelled follows from the fact that $(p,q)$-limited subsets of $E$ are exactly $(p,q)$-$\EE$-limited (see (i) of Lemma \ref{l:limited-set-1}).\qed
\end{proof}

We select the next theorem. % is similar to Theorem 16.8 of \cite{Gab-Pel}. %\ref{t:qp-convergent-pq-V-1}
\begin{theorem} \label{t:qp-convergent-pq-limited-1}
Let $1\leq p\leq q\leq \infty$, $E$ be a locally convex space, and let $T$ be an operator from a normed space $L$ to $E$. Then $T(B_L)$ is a $(p,q)$-limited subset of $E$ if and only if  $T^\ast:E'_{w^\ast} \to L'_\beta$ is $(q,p)$-convergent.
\end{theorem}

\begin{proof}
Observe that for every $\chi\in E'$, we have
\begin{equation} \label{equ:qp-convergent-pq-limited-1}
\|T^\ast(\chi)\|_{L'_\beta}=\sup_{y\in B_L} |\langle T^\ast(\chi),y\rangle|= \sup_{y\in B_L} |\langle \chi,T(y)\rangle|.
\end{equation}
Let $\{\chi_n\}_{n\in\w}$ be a weak$^\ast$ $p$-summable sequence in $E'$. Then, by (\ref{equ:qp-convergent-pq-limited-1}), we have
$\big(\|T^\ast(\chi_n)\|_{L'_\beta}\big)_{n\in\w} =\big( \sup_{y\in B_L} |\langle \chi_n, T(y)\rangle|\big)_{n\in\w}$. Now the theorem follows from the definition of $(p,q)$-limited sets and  the definition of $(q,p)$-convergent linear map.\qed
\end{proof}

It is natural to find some classes of subsets which are $(p,q)$-limited. Below, under additional assumption on an lcs $E$, we show that any precompact subset $A$ of $E$ is $(p,\infty)$-limited. %, cf. Proposition 7.6 in \cite{Gab-Pel}. %{p:precompact-p-V*}

\begin{proposition} \label{p:precompact-p-limited}
Let  $p\in[1,\infty]$, and let $E$  be a locally convex space.
\begin{enumerate}
\item[{\rm(i)}] Every precompact subset $A$ of $E$ is $(p,\infty)$-$\EE$-limited.
\item[{\rm(ii)}] If $E$ is $p$-barrelled, then every precompact subset $A$ of $E$ is $(p,\infty)$-limited.
\end{enumerate}
\end{proposition}

\begin{proof}
Let $S=\{\chi_n\}_{n\in\w}$ be a (resp., equicontinuous) weak$^\ast$ $p$-summable sequence  in $E'$. If $E$ is $p$-barrelled, then $S$ is equicontinuous. Therefore in both cases (i) and (ii) we can assume that $S$ is equicontinuous. Hence, by Proposition 3.9.8 of \cite{horvath}, the weak$^\ast$ topology $\sigma(E',E)$ and the topology $\tau_{pc}$ of uniform convergence on precompact subsets of $E$ coincide on $S$. Since $S$ is weak$^\ast$ $p$-summable, it is a weak$^\ast$ null-sequence. Therefore $\chi_n\to 0$ also in  $\tau_{pc}$. As $A$ is precompact, we obtain $\sup_{x\in A} |\langle \chi_n, x\rangle|\to 0$. Thus $A$ is a $(p,\infty)$-limited set (resp., a $(p,\infty)$-$\EE$-limited set).\qed
\end{proof}

Since, by definition, $\infty$-barrelled spaces are exactly $c_0$-barrelled, setting $p=\infty$ in (ii) of Proposition \ref{p:precompact-p-limited} we obtain the next assertion.

\begin{corollary} \label{c:precompact-p-limited}
If $E$ is a $c_0$-barrelled space, then every precompact subset of $E$ is limited.
\end{corollary}

The condition in (ii) of Proposition \ref{p:precompact-p-limited} that $E$ is $p$-barrelled  is essential as the following theorem shows.
Moreover, it may happen that a non-$p$-barrelled space contains even compact sets which are not limited, see Example \ref{exa:compact-not-limited} below.

%\begin{theorem} \label{t:b-feral-feral-L}
%Let $p\in[1,\infty]$, $X$ be a Tychonoff space, and let $\TTT$ be a locally convex vector topology on $L(X)$ compatible with the duality $(L(X),C(X))$. Then the following assertions are equivalent:
%\begin{enumerate}
%\item[{\rm(i)}] every  functionally bounded subset of $X$ is finite;
%\item[{\rm(ii)}] $L_\TTT(X)$ is quasi-complete;
%\item[{\rm(iii)}] $L_\TTT(X)$ is  sequentially complete;
%\item[{\rm(iv)}] $L_\TTT(X)$ is  locally complete.
%\item[{\rm(v)}]  $L_\TTT(X)$ is weakly locally $p$-complete.
%\end{enumerate}
%\end{theorem}

%\begin{proof}
%(i)$\Ra$(ii) Since all  functionally bounded subsets of $X$ are finite, Proposition 2.7 of \cite{Gabr-free-lcs} implies that any bounded subset of $L_\TTT(X)$ is finite-dimensional. Thus $L_\TTT(X)$ is quasi-complete.

%The implication (ii)$\Ra$(iii) and (iii)$\Ra$(iv) are trivial. %, and (iv)$\Ra$(v) follows from (iv) of Lemma \ref{l:weakly-p-lc}.

%(iv)$\Ra$(i) Assume that $L_\TTT(X)$ is locally complete. Since $(L(X),C(X))$ is a dual pair, Theorem \ref{t:p-loc-complete-p-bar} implies that the space $\big(C(X), \mu(C(X),L(X)\big)$ is $p$-barrelled. As $C_p(X)$ is quasibarrelled, we obtain that $\big(C(X), \mu(C(X),L(X)\big)=C_p(X)$ and hence $C_p(X)$ is $p$-barrelled. Thus, by Theorem \ref{t:Cp-p-barreled}, %\ref{t:Cp-p-barrelled},
%$X$ has no infinite functionally bounded sets. \qed
%\end{proof}
%For numerous  other equivalent conditions to (i)-(iv) of Theorem \ref{t:b-feral-feral-L}, see Theorem 3.5 of \cite{BG-free}.

\begin{theorem} \label{t:Cp-p-barreled}
Let $p\in[1,\infty]$, $X$ be a Tychonoff space, and let $\TTT$ be a locally convex vector topology on $L(X)$ compatible with the duality $(L(X),C(X))$. Then the following assertions are equivalent:
\begin{enumerate}
\item[{\rm(i)}] the space $C_p(X)$ is $p$-barrelled;
\item[{\rm(ii)}] every precompact (= bounded) subset of $C_p(X)$ is $(p,\infty)$-limited;
\item[{\rm(iii)}] $X$ has no infinite functionally bounded subsets;
\item[{\rm(iv)}] each bounded subset of $C_p(X)$ is $(p,q)$-limited for some (every) $p\leq q\leq\infty$;
\item[{\rm(v)}] $L_\TTT(X)$ is quasi-complete;
\item[{\rm(vi)}] $L_\TTT(X)$ is  sequentially complete;
\item[{\rm(vii)}] $L_\TTT(X)$ is  locally complete.
\end{enumerate}
\end{theorem}

\begin{proof}
(i)$\Ra$(ii) follows from (ii) of Proposition \ref{p:precompact-p-limited}.

%(i)$\Ra$(iii) Assume that $C_p(X)$ is $p$-barrelled, and suppose for a contradiction that $X$ has an infinite functionally bounded subset $A$. Then there are a sequence $\{x_n\}_{n\in\w}$ in $A$ and a sequence  $\{U_n\}_{n\in\w}$ of open subsets of $X$ such that $x_n\in U_n$ and $U_n\cap U_m=\emptyset$ for all distinct $n,m\in\w$. For every $n\in\w$, set $\chi_n:=\tfrac{1}{2^n} \delta_{x_n}$. Since $A$ is functionally bounded, we obtain that the sequence $S=\{\chi_n\}_{n\in\w}$ is weak$^\ast$ $p$-summable in $L(X)$. Since $C_p(X)$ is $p$-barrelled, we obtain that $S$  is equicontinuous. Therefore there are a finite subset $F$ of $X$ and $\e>0$ such that $S\subseteq [F;\e]^\circ$. Choose $m\in\w$ such that $m\not\in F$ and a continuous function $f:X\to[0,2^{m+1}]$ such that $f\big(F\cup (X\SM U_m)\big)\subseteq \{0\}$ and $f(x_m)=2^{m+1}$. Then $f\in [F;\e]^\circ$ and $\langle\chi_m,f\rangle=2>1$, and hence $\chi_m\not\in [F;\e]^\circ$. This is a contradiction.

(ii)$\Ra$(iii) Assume that every bounded subset of $C_p(X)$ is $(p,\infty)$-limited, and suppose for a contradiction that $X$ has an infinite functionally bounded subset $A$. Then one can find a sequence $\{x_n\}_{n\in\w}$ in $A$ and a sequence  $\{U_n\}_{n\in\w}$ of open subsets of $X$ such that $x_n\in U_n$ and $U_n\cap U_m=\emptyset$ for all distinct $n,m\in\w$. Set
\[
B:=\Big\{f\in C_p(X): f(U_n)\subseteq [0,2^{n+1}] \mbox{ for all $n\in\w$, and } f\big( X\SM \bigcup_{n\in\w} U_n\big) \subseteq \{0\}\Big\}.
\]
Then $B$ is a bounded subset of $C_p(X)$, and hence $B$ is $(p,\infty)$-limited. For every $n\in\w$, set $\chi_n:=\tfrac{1}{2^n} \delta_{x_n}$. Since $A$ is functionally bounded, we obtain that the sequence $S=\{\chi_n\}_{n\in\w}$ is weak$^\ast$ $p$-summable in the dual space $C_p(X)'$. For every $n\in\w$, take a continuous function $g_n:X\to[0,2^{n}]$ such that $g_n(X\SM U_n)\subseteq \{0\}$ and $g_n(x_n)=2^{n}$. It is clear that $g_n\in B$ for all $n\in\w$. However, since
\[
\sup_{f\in B} |\langle\chi_n,f\rangle|\geq |\langle\chi_n,g_n\rangle|=1 \not\to 0,
\]
we obtain that $B$ is not $(p,\infty)$-limited, a contradiction.
\smallskip

(iii)$\Ra$(i) Assume that $X$ has no infinite functionally bounded subsets. By the Buchwalter--Schmets theorem, the space $C_p(X)$ is barrelled and hence it is $p$-barrelled.

(iii)$\Ra$(iv) Fix $p\leq q\leq\infty$, and let $B$ be a bounded subset of $C_p(X)$. Take an arbitrary  weak$^\ast$ $p$-summable sequence $S=\{\chi_n\}_{n\in\w}$ in $C_p(X)'=L(X)$. Since $S$ is weak$^\ast$ bounded and the topology of the free lcs $L(X)$ is compatible with $\sigma\big( L(X),C_p(X)\big)$ it follows that $S$ is a bounded subset of $L(X)$. As all functionally bounded subsets of $X$ are finite, Proposition 2.7 of \cite{Gabr-free-lcs} implies that  $S$ is finite-dimensional.
By Lemma 4.6 of \cite{Gab-Pel}, %\ref{l:lp-finite-dim}
there are linearly independent elements $\eta_1,\dots,\eta_s\in L(X)$ and sequences $(a_{1,n}),\dots,(a_{s,n})\in \ell_p$ (or $\in c_0$ if $p=\infty$) such that
\[
\chi_n=a_{1,n} \eta_1 +\cdots +a_{s,n}\eta_s \;\; \mbox{ for every $n\in\w$}.
\]
Now, since $B$ is a bounded subset of $C_p(X)$ we obtain
\[
\sup_{f\in B} |\langle\chi_n,f\rangle| \leq \sum_{i=1}^s |a_{i,n}| \cdot \sup_{f\in B} |\langle\eta_i,f\rangle|
\]
and hence the inequality $p\leq q$ implies $\big(\sup_{f\in B} |\langle\chi_n, f\rangle|\big)_n \in\ell_q$ (or $\in c_0$ if $q=\infty$). Therefore $B$ is a $(p,q)$-limited set, as desired.
\smallskip

(iv)$\Ra$(ii)   Assume that each bounded subset of $C_p(X)$ is $(p,q)$-limited for some  $p\leq q\leq\infty$. Then, by (vi) of Lemma  \ref{l:limited-set-1}, every bounded subset of $C_p(X)$ is $(p,\infty)$-limited.

(iii)$\Ra$(v) Since all  functionally bounded subsets of $X$ are finite, Proposition 2.7 of \cite{Gabr-free-lcs} implies that any bounded subset of $L_\TTT(X)$ is finite-dimensional. Thus $L_\TTT(X)$ is quasi-complete.

The implication (v)$\Ra$(vi) and (vi)$\Ra$(vii) hold true for any lcs. %, and (iv)$\Ra$(v) follows from (iv) of Lemma \ref{l:weakly-p-lc}.

(vii)$\Ra$(i) Assume that $L_\TTT(X)$ is locally complete. Since $(L(X),C(X))$ is a dual pair and $L_\TTT(X)$ is locally complete, it follows that $L(X)_{w^\ast}$ is also locally complete. As $C_p(X)$ is quasibarrelled hence Mackey, Theorem 5.6 of \cite{Gab-Pel}, %\ref{t:loc-complete-p-quasi}
implies that $C_p(X)$ is $p$-barrelled.\qed %the space $\big(C(X), \mu(C(X),L(X)\big)$ is $p$-barrelled. As $C_p(X)$ is quasibarrelled, we obtain that $\big(C(X), \mu(C(X),L(X)\big)=C_p(X)$ and hence $C_p(X)$ is $p$-barrelled. \qed
\end{proof}
For numerous  other equivalent conditions to (i)-(vii) of Theorem \ref{t:Cp-p-barreled} %{t:b-feral-feral-L},
see Theorem 3.5 of \cite{BG-free}.

According to (viii) of Lemma \ref{l:limited-set-1}, every $(p,q)$-limited set is a $(p,q)$-$(V^\ast)$ set, but the converse is not true in general as the following corollary shows.
\begin{corollary} \label{c:Cp-V*-sets-non-GP}
Let $1\leq p\leq q\leq\infty$, and let $X$ be a Tychonoff space which has infinite functionally bounded subsets. Then $C_p(X)$ contains $(p,q)$-$(V^\ast)$ sets which are not $(p,q)$-limited.
\end{corollary}

\begin{proof}
By Corollary 7.11 of \cite{Gab-Pel}, %{c:Cp-V*-sets}
for every  Tychonoff space $X$ we have $\mathsf{V}_{(p,q)}^\ast\big(C_p(X)\big)=\Bo\big(C_p(X)\big)$. % follow from Propositions \ref{p:strong-dual-feral} and \ref{p:feral-V*-prop}.
Now the assertion follows from Theorem \ref{t:Cp-p-barreled}.\qed
\end{proof}

For a better understanding it is convenient to have a concrete example of a compact subset which is not limited. Denote by $\mathbf{s}=[0,\w]$ a convergent sequence.
\begin{example} \label{exa:compact-not-limited}
There are compact subsets of $C_p(\mathbf{s})$ which are not limited.
\end{example}

\begin{proof}
For every $n\in\w$, let $f_n=1_{\{n\}}$ be the characteristic function of the set $\{n\}$ and let $\chi_n :=\delta_n -\delta_{n+1}$, where $\delta_x$ denoted the Dirac measure at the point $x$. Evidently, the sequence $S=\{f_n\}_{n\in\w}$ is a null sequence in $C_p(\mathbf{s})$, and the sequence $\{\chi_n\}_{n\in\w}$ is weak$^\ast$ null. Since $\sup\{|\langle\chi_n, f_i\rangle|:i\in\w\}\geq |\langle\chi_n, f_n\rangle|=1 \not\to 0$ it follows that $S$ is not limited.\qed
\end{proof}

Proposition \ref{p:precompact-p-limited} motivates the following inverse problem: {\em Characterize locally convex spaces whose $(p,q)$-limited subset {\rm(}resp.,  $(p,q)$-$\EE$-limited subset{\rm)} of $E$ are precompact}. We solve this problem in the next theorem. % which is an analogue of Theorem 7.12 of \cite{Gab-Pel}. % \ref{t:V*-set-precompact}
\begin{theorem} \label{t:limited-set-precompact}
Let $1\leq p\leq q\leq\infty$. For a locally convex space $E$ the following assertions are equivalent:
\begin{enumerate}
\item[{\rm (i)}] every $(p,q)$-limited subset {\rm(}resp.,  $(p,q)$-$\EE$-limited subset{\rm)} of $E$ is precompact;
\item[{\rm (ii)}] each operator $T:L\to E$ from an lcs $L$ to $E$ which transforms bounded subsets of $L$ to  $(p,q)$-limited subsets {\rm(}resp.,  $(p,q)$-$\EE$-limited subset{\rm)}  of $E$, transforms bounded subsets of $L$ to precompact subsets of $E$;
\item[{\rm (iii)}] as in {\rm(ii)} with a normed space $L$.
\end{enumerate}
If in addition $E$ is locally complete, then {\rm(i)--(iii)} are equivalent to
\begin{enumerate}
\item[{\rm (iv)}] as in {\rm(ii)} with a Banach space $L$.
\end{enumerate}
\end{theorem}

\begin{proof}
(i)$\Rightarrow$(ii) Let $T:L\to E$ be an operator which transforms bounded subsets of an lcs $L$ to  $(p,q)$-limited (resp.,  $(p,q)$-$\EE$-limited) subsets of $E$. Let $A$ be a  bounded subset of $L$. Then $T(A)$ is a $(p,q)$-limited  (resp.,  $(p,q)$-$\EE$-limited) subset of $E$, and hence, by (i), $T(A)$ is precompact. Thus $T$ transforms bounded subsets of $L$ to precompact  subsets of $E$.
\smallskip

(ii)$\Rightarrow$(iii) and (ii)$\Rightarrow$(iv) are trivial.
\smallskip

(iii)$\Rightarrow$(i) and (iv)$\Rightarrow$(i): Fix a $(p,q)$-limited  (resp.,  $(p,q)$-$\EE$-limited)  subset $A$ of $E$. By (iii) of Lemma \ref{l:limited-set-1}, without loss of generality we can assume that $A=A^{\circ\circ}$. Consider the normed  space $E_A$ (if $E$ is locally complete, then $E_A$ is a Banach space), where the norm on $E_A$ is defined by the gauge of $A$, and recall that the closed unit ball $B_A$ of  $E_A$ is exactly $A$. By Propositions 3.2.2 and 5.1.6 of \cite{PB}, %Proposition \ref{p:bounded-norm},
the identity inclusion $T:E_A\to E$ is continuous and the set $T(B_A)=A$ is a $(p,q)$-limited  (resp.,  $(p,q)$-$\EE$-limited) set. Since any bounded subset of $E_A$ is contained in some $aB_A$, $a>0$, Lemma \ref{l:limited-set-1} implies that $T$ transforms bounded subsets of the normed (resp., Banach) space $E_A$ to  $(p,q)$-limited  (resp.,  $(p,q)$-$\EE$-limited)  subsets of $E$.  Therefore, by (iii) or (iv), the set $A=T(B)$ is precompact.\qed
\end{proof}

By (vi) of Lemma \ref{l:limited-set-1}, every $(p,q)$-limited set is $(1,\infty)$-limited. Therefore to characterize  $(1,\infty)$-limited sets in locally convex spaces is an important problem. For barrelled spaces we solve this problem in Proposition \ref{p:char-1-limited} below.
Our proof is similar to the proof of Proposition 1.1 of \cite{Bombal}, where it is obtained a characterization of $(V^\ast)$ sets in Banach spaces. First we prove the next lemma.
\begin{lemma} \label{l:1-limited-l1}
If a subset $A$ of a locally convex space $E$ is a $(1,\infty)$-limited set, then  $T(A)$ is relatively compact for every operator $T:E\to \ell_1$.
\end{lemma}

\begin{proof}
Suppose for a contradiction that $T(A)$ is not relatively compact in $\ell_1$ for some operator $T:E\to \ell_1$. By (i) of Proposition 4.17 of \cite{Gab-Pel}, %\ref{p:operator-Lp},
$T(x)=(\langle \chi_n,x\rangle)_{n\in\w}$ for some equicontinuous weak$^\ast$ $1$-summable sequence $\{\chi_n\}_{n\in\w}$ in $E'$. Then Proposition \ref{p:compact-ell-p}
implies that there are $\e>0$, a sequence $r_0<s_0<r_1<s_1<\cdots$ in $\w$, and a sequence $\{a_j\}_{j\in\w}$ in $A$ such that
\[
\sum_{n=r_j}^{s_j} |\langle \chi_n,a_j\rangle| >\e \;\; \mbox{ for every } \; j\in\w.
\]
For every $j\in\w$, by Lemma 6.3 of \cite{Rudin},  there is a subset $F_j$ of $[r_j,s_j]$ such that
\begin{equation} \label{equ:V*-1}
\Big| \sum_{n\in F_j} \langle \chi_n,a_j\rangle\Big| >\tfrac{\e}{4}.
\end{equation}
For every $j\in\w$, set $\eta_j:= \sum_{n\in F_j}\chi_n$. Then the sequence $\{\eta_j\}_{j\in\w}$ is weak$^\ast$ $1$-summable in $E'$. By (\ref{equ:V*-1}), we have
\[
\sup_{a\in A} |\langle \eta_j,a\rangle|\geq |\langle \eta_j,a_j\rangle|>\tfrac{\e}{4},\quad \mbox{ for every }\; j\in\w,
\]
and hence $A$ is not a $(1,\infty)$-limited set, a contradiction.\qed
\end{proof}

%The following assertion generalizes Proposition 1.1 of \cite{Bombal} to barrelled locally convex spaces.%, Corollary 20 of \cite{Ghenciu-pGP} and Proposition 3.1 of \cite{Ghenciu-weakDP}.

%{t:image-l1}
\begin{proposition} \label{p:char-1-limited}
For a bounded subset $A$ of a barrelled locally convex space $E$ the following assertions are equivalent:
\begin{enumerate}
\item[{\rm(i)}] $A$ is a $(1,\infty)$-limited set;%$1$-$(V^\ast)$ set;
\item[{\rm(ii)}] $T(A)$ is relatively compact for every operator $T:E\to \ell_1$;
\item[{\rm(iii)}] for any weak$^\ast$ $1$-summable sequence $\{\chi_n\}_{n\in\w}$ in $E'$, it follows
\[
\lim_{m\to\infty} \sup\Big\{ \sum_{m\leq n} |\langle \chi_n,x\rangle|: x\in A\Big\}=0.
\]
\end{enumerate}
\end{proposition}

\begin{proof}
(i)$\Ra$(ii) follows from Lemma \ref{l:1-limited-l1}.
\smallskip

(ii)$\Ra$(iii) Let $\{\chi_n\}_{n\in\w}$ be a weak$^\ast$ $1$-summable sequence in $E'$. Then, by Proposition 4.19 of \cite{Gab-Pel}, %\ref{p:p-sum-operator},
there is an operator $T:E\to \ell_1$ such that $T(x):=(\langle\chi_n,x\rangle)$ for every $x\in E$. By (ii), the set $T(A)$ is relatively compact in $\ell_1$. %(if $T(A)$ is relatively weakly compact, then, by the Schur property of $\ell_1$, $T(A)$ is relatively compact in $\ell_1$)
Therefore, by Proposition \ref{p:compact-ell-p},
we obtain
\[
\lim_{m\to\infty} \sup\Big\{ \sum_{m\leq n} |y_n| : y=(y_n)\in T(A)\Big\} =0.
\]
It remains to note that if $y=(y_n)=T(x)$ for some $x\in A$, then $y_n=\langle\chi_n,x\rangle$ for all $n\in\w$.
\smallskip

(iii) implies (i) since $\sup_{a\in A} |\langle \chi_m,a\rangle| \leq\sup\big\{ \sum_{m\leq n} |\langle \chi_n,x\rangle|: x\in A\big\}\to 0$ for every  weak$^\ast$ $1$-summable sequence $\{\chi_n\}_{n\in\w}$ in $E'$.\qed
\end{proof}

\begin{corollary} \label{c:weakly-seq-V*}
Weakly sequentially precompact subsets and precompact subsets of a barrelled locally convex space $E$ are $(1,\infty)$-limited.
\end{corollary}

\begin{proof}
Let $A$ be a weakly sequentially precompact subset of $E$  or a precompact subset of $E$. Then for every operator $T:E\to \ell_1$, the image $T(A)$ is also weakly sequentially precompact or precompact in $\ell_1$. By Lemma \ref{l:seq-precom-Schur}, $T(A)$ is (sequentially) precompact and hence relatively compact in $\ell_1$. Thus, by Proposition \ref{p:char-1-limited}, $A$ is  a $(1,\infty)$-limited set.\qed
\end{proof}
%Note that there are precompact subsets of the complete barrelled space $\IR^\mathfrak{c}$ which are not sequentially precompact.

It is natural to characterize spaces for which {\em all} relatively weakly sequentially $p$-compact sets are $(q,\infty)$-limited. For the case when $E$ is a Banach space and $q=\infty$, the next proposition gives $(a)\Leftrightarrow(b)$ of Proposition 2.10 of \cite{FZ-p}.

\begin{proposition} \label{p:ws-p-compact-q-limited}
Let $1\leq p\leq q\leq \infty$. For  a locally convex space $E$, the following assertions are equivalent:
\begin{enumerate}
\item[{\rm(i)}] all relatively weakly sequentially $p$-compact subsets of $E$ are $(q,\infty)$-limited {\rm(}resp., $(q,\infty)$-$\EE$-limited{\rm)};
\item[{\rm(ii)}] every weakly $p$-summable sequence in $E$ is $(q,\infty)$-limited {\rm(}resp., $(q,\infty)$-$\EE$-limited{\rm)}.
\end{enumerate}
\end{proposition}

\begin{proof}
(i)$\Ra$(ii) is clear because every weakly $p$-summable sequence is relatively weakly sequentially $p$-compact.

(ii)$\Ra$(i) Suppose for a contradiction that there is a relatively weakly sequentially $p$-compact subset $A$ of $E$ which is not  $(q,\infty)$-limited (resp., $(q,\infty)$-$\EE$-limited). Then there are a weak$^\ast$ (resp., equicontinuous) $q$-summable sequence $\{\chi_n\}_{n\in\w}$ in $E'$ and $\e>0$ such that $\sup_{a\in A} |\langle\chi_n,a\rangle|\geq \e$ for every $n\in\w$. For every $n\in\w$, choose $a_n\in A$ such that
\begin{equation} \label{equ:ws-p-compact-q-limited-1}
|\langle\chi_n,a_n\rangle|\geq \tfrac{\e}{2}.
\end{equation}
Since $A$ is relatively weakly sequentially $p$-compact, there are $a\in E$ and  a subsequence $\{a_{n_k}\}_{k\in\w}$ of $\{a_n\}_{n\in\w}$ such that $\{a_{n_k}-a\}_{n\in\w}$ is weakly $p$-summable. Taking into account that $\{\chi_{n_k}\}_{k\in\w}$ is also weak$^\ast$ null, (ii) and (\ref{equ:ws-p-compact-q-limited-1}) imply
\[
\tfrac{\e}{2}\leq |\langle\chi_{n_k},a_{n_k}\rangle|\leq |\langle\chi_{n_k},a_{n_k}-a\rangle|+|\langle\chi_{n_k},a\rangle|\leq \sup_{i\in\w} |\langle\chi_{n_k},a_{n_i}-a\rangle|+|\langle\chi_{n_k},a\rangle|\to 0,
\]
a contradiction. \qed
\end{proof}

If $E$ is a Banach space and $p=\infty$, the next theorem was proved by Grothendieck in \cite{Grothen}. % (see also Theorem 3.11 of \cite{HMVZ}).
Note that Example 3.5 of \cite{Gab-DP} %\ref{exa:F-p-angelic}
 shows that our result is indeed more general than the Grothenfieck theorem.
\begin{theorem} \label{t:weakly-mu-p-limited}
Let $p\in[1,\infty]$, $E$ be a weakly  $p$-angelic and locally complete space, and let $H:=\big(E',\mu(E',E)\big)$. Then a bounded subset $B$ of $H$ is precompact if and only if it is a $(p,\infty)$-limited set.
\end{theorem}

\begin{proof}
%Since $E$ is weakly $p$-locally complete, Proposition \ref{p:p-loc-complete-p-bar} implies that $H$ is $p$-barrelled. Therefore, by Proposition \ref{p:precompact-p-limited}, every precompact subset of $H$ is $p$-limited.
Assume that $B$ is a $\mu(E',E)$-precompact subset of $E'$. Let $S=\{x_n\}_{n\in\w}$ be a weak$^\ast$ $p$-summable sequence in $H'=E$. Then $S$ is weakly $p$-summable in $E$. Since $E$ is locally complete, the set $K:=\cacx(S)$ is weakly compact, and hence, by the Mackey--Arens theorem, $K^\circ$ is a neighborhood of zero in $H$. Consequently, $K=K^{\circ\circ}$ and hence also $S$ are $\mu(E',E)$-equicontinuous. The weak$^\ast$ $p$-summability of $S$ implies that $S$ is weak$^\ast$ null. Hence, by Proposition 3.9.8 of \cite{horvath}, $x_n\to 0$ in the topology of uniform convergence on precompact subsets of $H$, in particular, $x_n\to 0$ uniformly on $B$. Since $S$ was arbitrary, by definition this means that $B$ is a $(p,\infty)$-limited subset of $H$.
\smallskip

Conversely, assume that $B$ is a $(p,\infty)$-limited subset of $H$. Let $u=\Id_E:E\to E$ be the identity map, $\mathfrak{S}$ be the family of all absolutely convex weakly compact subsets of $E$, $\mathfrak{T}$ be the family of all $(p,\infty)$-limited subsets of $H$. Then the equivalence of (1) and (1') in Theorem 12 of \cite[p.~91]{Grothendieck} can be formulated as follows:  the $(p,\infty)$-limited sets of $H$ are precompact if and only if any set $K\in \mathfrak{S}$ is precompact for the topology $\TTT$ of uniform convergence on all $(p,\infty)$-limited sets of $H$. Therefore to prove that $B$ is precompact it suffices to show that any  $K\in \mathfrak{S}$ is $\TTT$-precompact. To this end, fix a $K\in \mathfrak{S}$.

We claim that the topology $\TTT$ and the weak topology $\sigma(E,E')$ coincide on $K$. By (xi) of Lemma \ref{l:limited-set-1}, we have $\sigma(E,E')\subseteq \TTT$. Therefore to prove the claim we have to show only that any $\TTT$-closed subset $A$ of $K$ is also weakly closed. Let $z\in \overline{A}^{\,\sigma(E,E')}$. Since $E$ is a  weakly  $p$-angelic space and $K$ is weakly compact, (i) of Lemma 3.6 of \cite{Gab-DP} %\ref{l:ws-p-angelic}
implies that there is a sequence $\{x_n\}_{n\in\w}$ in $A$ which weakly $p$-converges to $z$, i.e., the sequence $\{x_n-z\}_{n\in\w}$ is weakly $p$-summable. By the definition of $(p,\infty)$-limited sets we obtain that $x_n\to z$ in the topology $\TTT$ and hence $z\in A$. Thus $A$ is weakly closed. The claim is proved.

%Since $K$ is weakly compact, Proposition 3.4 and Corollary 3.7 of \cite{HMVZ} and the claim imply that $K$ is precompact for the topology $\TTT$.\qed
Since $K$ is weakly compact, the claim implies that also $K$ is compact for the topology $\TTT$.\qed
\end{proof}

It is convenient to formulate Theorem \ref{t:weakly-mu-p-limited} in a dual form.

\begin{corollary} \label{c:weakly-mu-p-limited-dual}
Let $p\in[1,\infty]$, and let $E$ be a Mackey space. If the space $E'_{w^\ast}$ is weakly $p$-angelic and  locally complete, then a bounded subset $A$ of $E$ is precompact if and only if it is a $(p,\infty)$-limited set.
\end{corollary}

\begin{proof}
Set $E_1:=E'_{w^\ast}$ (so $E_1$ carries its weak topology) and $H_1:=\big(E_1', \mu(E'_1,E_1)\big)$. Then $E'_1=E$ algebraically. Since, by the Mackey--Arens theorem, the polars of the weak$^\ast$ compact absolutely convex subsets of $E'_{w^\ast}=E_1$ define the Mackey topology on $E$ and the Mackey topology $\mu(E'_1,E_1)$ on $E'_1=E$, we obtain $\mu(E'_1,E_1)=\mu(E,E')$. As $E$ is a Mackey space it follows  $H_1=E$. Now Theorem  \ref{t:weakly-mu-p-limited} applies.\qed
\end{proof}

Theorem 5.6 of \cite{Gab-Pel} %\ref{t:loc-complete-p-quasi}
(which states that a Mackey space $E$ is $p$-barrelled if and only if $E'_{w^\ast}$ is locally complete) and Corollary \ref{c:weakly-mu-p-limited-dual} imply
\begin{corollary} \label{c:weakly-mu-p-ld}
Let $p\in[1,\infty]$, and let $E$ be a Mackey $p$-barrelled space {\rm(}for example, $E$ is barrelled{\rm)}. If $E'_{w^\ast}$ is a weakly $p$-angelic space, then a bounded subset $A$ of $E$ is precompact if and only if it is a $(p,\infty)$-limited set.
\end{corollary}

The case $p=\infty$ is of independent interest.
\begin{corollary} \label{c:weakly-mu-inf-limited}
Let a locally convex space $E$ satisfy one of the following conditions:
\begin{enumerate}
\item[{\rm(i)}] $E$ is a Mackey $c_0$-barrelled space such that $E'_{w^\ast}$ is a weakly angelic space;
\item[{\rm(ii)}] $E$ is a reflexive space  such that $E'_{\beta}$ is a weakly angelic space;
\item[{\rm(iii)}] $E$ is a separable Mackey $c_0$-barrelled space.
\end{enumerate}
Then a bounded subset $A$ of $E$ is precompact if and only if it is a limited set. Moreover, if in addition $E$ is von Neumann complete, then a bounded subset $A$ of $E$ is  relatively compact if and only if it is a limited set.
\end{corollary}

\begin{proof}
(i) Proposition 3.4 of \cite{Gab-DP} %\ref{p:angelic-p-ang}
states that every weakly angelic space $E$ is weakly $\infty$-angelic. Now Corollary \ref{c:weakly-mu-p-ld} applies. %Since weakly $\infty$-summable sequences are exactly weakly null sequences, it follows that relatively weakly sequentially $\infty$-compact sets are exactly relatively weakly sequentially compact sets. Therefore $E'_{w^\ast}$ is a weakly sequentially $\infty$-angelic space. Now Corollary \ref{c:weakly-mu-p-limited-dual} applies.

(ii) Recall that any reflexive space is barrelled,  see Proposition 11.4.2 of \cite{Jar}. Since $E$ is semi-reflexive, $E'_{\beta}$ is a weakly angelic space if and only if so is  $E'_{w^\ast}$. Now (i) applies.

(iii) Since $E$ is separable, the space $E'_{w^\ast}$ admits a weaker metrizable locally convex topology $\TTT$. Therefore $(E',\TTT)$ and hence also $E'_{w^\ast}$ are even (weakly) angelic spaces. By Proposition 3.4 of \cite{Gab-DP}, %\ref{p:angelic-p-ang}
the space $E'_{w^\ast}$ is  weakly (sequentially) $\infty$-angelic. Now (i) applies.

The last assertion follows from the fact that if in addition $E$ is von Neumann complete, then any precompact subset of $E$ is relatively compact.\qed
\end{proof}

%%%%%%%%%%%%%%%%%%%%%%%%%%%%%%%%%%%
%%%%%%%%%%%%%%%%%%%%%%%%%%%%%%%%%%%
%%%%%%%%%%%%%%%%%%%%%%%%%%%%%%%%%%%
%%%%%%%%%%%%%%%%%%%%%%%%%%%%%%%%%%%
%%%%%%%%%%%%%%%%%%%%%%%%%%%%%%%%%%%

\section{Coarse $p$-limited sets} \label{sec:coarse}

%%%%%%%%%%%%%%%%%%%%%%%%%%%%%%%%%%%
%%%%%%%%%%%%%%%%%%%%%%%%%%%%%%%%%%%
%%%%%%%%%%%%%%%%%%%%%%%%%%%%%%%%%%%
%%%%%%%%%%%%%%%%%%%%%%%%%%%%%%%%%%%
%%%%%%%%%%%%%%%%%%%%%%%%%%%%%%%%%%%

Below we summarize some basic properties of coarse $p$-limited sets, cf. Proposition 2 of \cite{GalMir}.
\begin{lemma} \label{l:coarse-p-limited}
Let $p,q\in[1,\infty]$, and let $(E,\tau)$ be a locally convex space. Then:
\begin{enumerate}
\item[{\rm(i)}] every coarse $p$-limited subset of $E$ is bounded;
\item[{\rm(ii)}] the family $\mathsf{CL}_p(E)$ of all coarse $p$-limited  sets in $E$ is closed under taking subsets, finite unions and sums, and closed absolutely convex hulls;
\item[{\rm(iii)}] if $L$ is a locally convex space and $T\in \LL(E,L)$ and if $A\subseteq E$ is coarse $p$-limited, then $T(A)$ is  a coarse $p$-limited subset of $L$;
\item[{\rm(iv)}] a  subset $A$ of $E$ is a coarse $p$-limited set  if and only if every countable subset of $A$ is a coarse $p$-limited set;
\item[{\rm(v)}] {\rm(\cite[Remark~2]{GalMir})} in general, even for Banach spaces there is no inclusion relationships between the class of coarse $p$-limited sets and  the class of coarse $q$-limited sets for $p\not= q$;
\item[{\rm(vi)}] in general the property of being a coarse $p$-limited set is not the property of  the duality $(E,E')$.
\end{enumerate}
\end{lemma}

\begin{proof}
The clauses (i)-(iii) are clear. %If $\chi\in E'$ such that $\sup_{a\in A} |\langle\chi,A\rangle|=\infty$, define an operator $T:E\to \ell_p$ by $T(x):=(\langle\chi,x\rangle,0,0,\dots)$. Then $T(A)$ is unbounded, a contradiction.

(iv) The necessity follows from (ii). To prove the sufficiency we consider the case $1\leq p<\infty$ since the case $p=\infty$ can be considered analogously. Suppose for a contradiction that $A$ is not a coarse $p$-limited  set in $E$.  Then there is an operator $T:E\to\ell_p$ such that $T(A)$ is not relatively compact in $\ell_p$. Then, by Proposition \ref{p:compact-ell-p}, there is $\e>0$ such that
\[
\sup_{ a\in A}\Big\{ \sum_{m\leq n} |\langle T^\ast(e^\ast_n), a\rangle|^p\Big\} \geq \e \;\; \mbox{ for every } m\in\w.
\]
For every $m\in\w$, choose $a_m\in A$ such that $\sum_{m\leq n} |\langle T^\ast(e^\ast_n), a_m\rangle|^p>\tfrac{\e}{2}$. By assumption the sequence $\{a_m\}_{m\in\w}$ is a coarse $p$-limited set. Therefore, by Proposition \ref{p:compact-ell-p} and the choice of $a_m$, we have
\[
\tfrac{\e}{2}<\sup_{m\in \w}\Big\{ \sum_{m\leq n} |\langle T^\ast(e^\ast_n), a_m\rangle|^p\Big\} \to 0,
\]
a contradiction.
\smallskip

(vi) Let $1\leq p<\infty$, and let $E=\ell_p$ (for $p=\infty$, one can consider $E=c_0$). Then the unit ball $B_E$ is not a coarse $p$-limited set in $E$ (if $\Id:E\to \ell_p$ is the identity map then $\Id(B_E)$ is not relatively compact in $\ell_p$). However, since every $T\in \LL(E_w,\ell_p)$ is finite-dimensional by Lemma 17.18 of \cite{Gab-Pel} %\ref{l:weak-to-norm}
it follows that $B_E$ is a coarse $p$-limited set in $E_w$. \qed
\end{proof}

If $1\leq p<\infty$, Proposition 1 of \cite{GalMir} states that every $p$-limited subset of a Banach space is a coarse $p$-limited set. Below we generalize this result.
\begin{proposition} \label{p:pV*-coarse-p-lim}
Let $E$ be a locally convex space. Then:
\begin{enumerate}
\item[{\rm(i)}] if $p\in[1,\infty]$ and $S_p\big(\LL(E,\ell_p)\big)\subseteq\ell_p^w(E'_\beta)$, then every $(p,p)$-$(V^\ast)$ subset of $E$ is a coarse $p$-limited set;
\item[{\rm(ii)}]  if $1<p<\infty$, then every $(p,p)$-$(V^\ast)$ subset $A$ of $E$ is a coarse $p$-limited set;
\item[{\rm(iii)}] if $p=\infty$ and $S_\infty\big(\LL(E,c_0)\big)=c_0^w(E'_\beta)$, then the class of $\infty$-$(V^\ast)$ subsets of $E$ coincides with the class of coarse $\infty$-limited subsets of $E$;
\item[{\rm(iv)}] if $p\in[1,\infty]$, then  every $p$-limited subset of $E$ is  coarse $p$-limited; in particular, every finite subset of $E$ is coarse $p$-limited.
\end{enumerate}
\end{proposition}

\begin{proof}
(i) Let $A$ be a $(p,p)$-$(V^\ast)$ subset of $E$, and let $T:E\to \ell_p$ (or $T:E\to c_0$ if $p=\infty$) be an operator. For every $n\in\w$, we set $\chi_n:=T^\ast(e^\ast_n)$. Then the inclusion $S_p\big(\LL(E,\ell_p)\big)\subseteq\ell_p^w(E'_\beta)$ implies that the sequence $\{\chi_n\}_n$ is weakly $p$-summable in $E'_\beta$. Since $A$ is a $(p,p)$-$(V^\ast)$ set, it follows that
\begin{equation} \label{equ:pV*-coarse-p-lim-1}
\big( \sup_{a\in A}|\langle\chi_n,a\rangle|\big) \in \ell_p \;\mbox{ (or $\in c_0$ if $p=\infty$). }
\end{equation}
Assume that $p<\infty$. Then (\ref{equ:pV*-coarse-p-lim-1}) implies
\[
\sup\left\{ \sum_{n=m}^\infty |\langle e^\ast_n,T(a)\rangle|^p: a\in A\right\} \leq \sum_{n=m}^\infty\Big(\sup_{a\in A}|\langle e^\ast_n,T(a)\rangle|\Big)^p <\infty.
\]
Therefore, by (i) of Proposition \ref{p:compact-ell-p}, $T(A)$ is relatively compact in $\ell_p$. Thus $A$ is  a coarse $p$-limited set.

If $p=\infty$, then (\ref{equ:pV*-coarse-p-lim-1}) yields $\lim_{n\to\infty} \sup\big\{|\langle e^\ast_n,T(a)\rangle|: a\in A\big\}=0$. Hence, by (ii) of Proposition \ref{p:compact-ell-p}, $T(A)$ is relatively compact in $c_0$. Thus $A$ is  a coarse $\infty$-limited set.
\smallskip

(ii) Assume that $1<p<\infty$. Then, by (iii) of Proposition 4.17 of \cite{Gab-Pel}, %\ref{p:operator-Lp},
$S_p\big(\LL(E,\ell_p)\big)\subseteq\ell_p^w(E'_\beta)$. Thus, by (i), $A$ is a coarse $p$-limited set.
\smallskip

(iii) Taking account (i), we have to prove that every coarse $\infty$-limited subset $A$ of $E$ is an $\infty$-$(V^\ast)$ set. To this end, let $\{\chi_n\}_n$ be a weakly $\infty$-summable sequence in $E'_\beta$. Then the equality $S_\infty\big(\LL(E,c_0)\big)=c_0^w(E'_\beta)$ implies that there is $T\in\LL(E,c_0)$ such that $\chi_n=T^\ast(e^\ast_n)$ for every $n\in\w$. Since $A$ is a  coarse $\infty$-limited set we obtain that $T(A)$ is a relatively compact subset of $c_0$. Therefore, by (ii) of Proposition \ref{p:compact-ell-p}, we have
\[
\sup\big\{ |\langle \chi_n,a\rangle|: a\in A\big\}= \sup\big\{ |\langle e^\ast_n,T(a)\rangle|: a\in A\big\}\to 0 \; \mbox{ as $n\to\infty$},
\]
which means that $A$ is an $\infty$-$(V^\ast)$ set.
\smallskip

(iv) Let $A$ be a $p$-limited subset of $E$, and let $T:E\to \ell_p$ (or $T:E\to c_0$ if $p=\infty$) be an operator. For every $n\in\w$, we set $\chi_n:=T^\ast(e^\ast_n)$. Then, by Proposition 4.17(i) of \cite{Gab-Pel}, %\ref{p:operator-Lp}(i)
the sequence $\{\chi_n\}_n$ is weakly $p$-summable in $E'_{w^\ast}$. Now proceeding exactly as in (i) we obtain that $A$ is a coarse $p$-limited subset of $E$.

For the last assertion it suffices to note that, by (xi) of Lemma \ref{l:limited-set-1}, %\ref{p:V*-q<p}
every finite subset of $E$ is a $p$-limited set.\qed
\end{proof}

\begin{proposition} \label{p:product-coarse-p-lim}
Let  $p\in[1,\infty]$, and let $\{E_i\}_{i\in I}$  be a non-empty family of locally convex spaces. Then:
\begin{enumerate}
\item[{\rm(i)}] a subset $K$ of $E=\prod_{i\in I} E_i$ is a coarse $p$-limited set if and only if so are all its coordinate projections;
\item[{\rm(ii)}]  a subset $K$ of  $E=\bigoplus_{i\in I} E_i$  is a coarse $p$-limited set if and only if so are all its coordinate projections   and the support of $K$ is finite.
\end{enumerate}
\end{proposition}

\begin{proof}
The necessity follows from (iii) of Lemma \ref{l:coarse-p-limited} since $E_i$ is a direct summand of $E$ and, for the case (ii), the well known fact that any bounded subset of a direct locally convex sum has finite support.
\smallskip

% Lemma \ref{l:support-p-sum}

To prove the sufficiency, let $K$ be a subset of $E$ such that each coordinate projection $K_i$ of $K$ is a  coarse $p$-limited  set in $E_i$, and, for the case (ii), $K_i=\{0\}$ for all but finitely many indices $i\in I$. By (ii) of Lemma \ref{l:coarse-p-limited}, %and (i) of Proposition \ref{p:pV*-coarse-p-lim}
 we can assume that $0\in K$. We distinguish between the cases (i) and (ii).
\smallskip

(i) Let $T:E\to\ell_p$ (or $T:E\to c_0$ if $p=\infty$) be an operator. It is easy to show (see for example Lemma 2.6 in \cite{Gab-Pel}) %\ref{l:product-normed},
that there is a finite subset $F$ of $I$ such that $\{0\}^F\times \prod_{i\in I\SM F} E_i$ is in the kernel of $T$. Then, taking into account that $0\in K$, we obtain $T(K)\subseteq \sum_{i\in F} T(K_i)$. Since, by assumption, all $T(K_i)$ are relatively compact in the Banach space $\ell_p$ (or in $c_0$) it follows that $\sum_{i\in F} T(K_i)$ (we identify $K_i$ with $K_i\times \prod_{I\SM \{i\}} \{0_i\}$) and hence also $T(K)$ are relatively compact in $\ell_p$ (or $c_0$). Thus $K$ is a  coarse $p$-limited  set in $E$.
\smallskip

(ii) Let $F\subseteq I$ be the finite support of $K$. Then $T(K)\subseteq \sum_{i\in F} T(K_i)$. As above in (i), it follows that  $T(K)$ is relatively compact in $\ell_p$ (or $c_0$). Thus $K$ is a  coarse $p$-limited  set in $E$. \qed
\end{proof}

%%%%%%%%%%%%%%%%%%%%%%%%%%%%%%%%%%%%%%%%
%%%%%%%%%%%%%%%%%%%%%%%%%%%%%%%%%%%%%%%%
%%%%%%%%%%%%%%%%%%%%%%%%%%%%%%%%%%%%%%%%
%%%%%%%%%%%%%%%%%%%%%%%%%%%%%%%%%%%%%%%%
%%%%%%%%%%%%%%%%%%%%%%%%%%%%%%%%%%%%%%%%

\section{Limited type sets and $p$-convergent operators} \label{sec:small-bound-p-conv}

%%%%%%%%%%%%%%%%%%%%%%%%%%%%%%%%%%%%%%%%
%%%%%%%%%%%%%%%%%%%%%%%%%%%%%%%%%%%%%%%%
%%%%%%%%%%%%%%%%%%%%%%%%%%%%%%%%%%%%%%%%
%%%%%%%%%%%%%%%%%%%%%%%%%%%%%%%%%%%%%%%%
%%%%%%%%%%%%%%%%%%%%%%%%%%%%%%%%%%%%%%%%

%The following assertion generalizes Theorem 14 of \cite{Ghenciu-pGP}.

Let  $1\leq p\leq q\leq\infty$, and let $E$ be a locally convex space. By Lemma 7.2 of \cite{Gab-Pel} %Lemma \ref{l:V*-set-1}
and Lemmas  \ref{l:limited-set-1} and \ref{l:coarse-p-limited}, the family $\mathsf{V}^\ast_{(p,q)}(E)$ of all $(p,q)$-$(V^\ast)$ sets, the family $\mathsf{L}_{(p,q)}(E)$ of all $(p,q)$-limited sets and the family $\mathsf{CL}_{p}(E)$ of all coarse $p$-limited sets in  $E$ are saturated bornologies. Therefore one can naturally define the following polar topologies on the dual space $E'$.

\begin{definition} \label{def:V-L-topology} {\em
Let $1\leq p\leq q\leq\infty$, and let $E$ be a locally convex space. Denote by $V^\ast_{(p,q)}(E',E)$ (resp., $EV^\ast_{(p,q)}(E',E)$, $L_{(p,q)}(E',E)$, $EL_{(p,q)}(E',E)$ and $CL_{p}(E',E)$) the polar topology on $E'$ of uniform convergence on $(p,q)$-$(V^\ast)$ (resp., $(p,q)$-$(EV^\ast)$, $(p,q)$-limited, $(p,q)$-$\EE$-limited or coarse $p$-limited) subsets of $E$.\qed}
\end{definition}
Since the families $\mathsf{V}^\ast_{(p,q)}(E)$ and $\mathsf{L}_{(p,q)}(E)$  depend only on the duality, the topologies $V^\ast_p(E',E)$  and $L_{(p,q)}(E',E)$ are topologies of the dual pair $(E,E')$. However, (vi) of Lemma \ref{l:coarse-p-limited} shows that the topology $CL_{p}(E',E)$ is not a topology of $(E,E')$. By this reason in what follows we consider only the topologies $\mathsf{V}^\ast_{(p,q)}(E)$ and $\mathsf{L}_{(p,q)}(E)$.

For further references we select the next simple lemma.

\begin{lemma} \label{l:V*-topology}
Let  $1\leq p\leq q\leq\infty$, and let $E$ be a locally convex space. Then:
\begin{enumerate}
\item[{\rm(i)}] $\sigma(E',E)\subseteq L_{(p,q)}(E',E) \subseteq V^\ast_{(p,q)}(E',E) \subseteq \beta(E',E)$,
%\item[{\rm(ii)}] $V^\ast_p(E',E) =\beta(E',E)$ if and only if $\mathsf{V}^\ast_p(E)=\Bo(E)$;
\item[{\rm(ii)}] $L_{(p,q)}(E',E) \subseteq \mu(E',E)$ if and only if every $(p,q)$-limited set $A$ in $E$ is relatively weakly compact;
\item[{\rm(iii)}] $L_{(p,q)}(E',E) = \mu(E',E)$ if and only if every $(p,q)$-limited set $A$ in $E$ is relatively weakly compact and every weakly compact absolutely convex subset of $E$ is $(p,q)$-limited.
\end{enumerate}
\end{lemma}

\begin{proof}
(i) follows from (viii) and (xi) of Lemma \ref{l:limited-set-1}. %follows from the fact that any $p$-$(V^\ast)$ set is bounded and that the families $\mathsf{V}^\ast_p(E)$ and $\Bo(E)$ are saturated bornologies.

(ii) and (iii) follow from the Mackey--Arens theorem and the fact that $\mathsf{L}_{(p,q)}(E)$ is a saturated bornology (see (iii) of Lemma \ref{l:limited-set-1}). \qed
\end{proof}

\begin{remark} \label{rem:V*-not-L-top} {\em
The inclusion $L_{(p,q)}(E',E) \subseteq V^\ast_{(p,q)}(E',E)$ can be strict. Indeed, let $X$ be a Tychonoff space containing an infinite functionally bounded subset. Then, by Corollary \ref{c:Cp-V*-sets-non-GP}, the space $C_p(X)$ contains $(p,q)$-$(V^\ast)$ sets which are not $(p,q)$-limited. This fact,  the inclusion $L_{(p,q)}(E',E) \subseteq V^\ast_{(p,q)}(E',E)$ and the fact that $\mathsf{V}_{(p,q)}^\ast\big(C_p(X)\big)$ and $\mathsf{L}_{(p,q)}^\ast\big(C_p(X)\big)$ are saturated bornologies imply that $L_{(p,q)}(E',E) \subsetneq V^\ast_{(p,q)}(E',E)$.\qed}
\end{remark}

It is well known that if $T\in\LL(E,L)$, then $T^\ast$ is weak$^\ast$ and strongly continuous. The following assertion shows that $T^\ast$ is also continuous with respect to the topology $L_{(p,q)}$.

%The next result naturally complements Theorem \ref{t:bounded-to-p-V*}.
\begin{proposition} \label{p:T*-p-convergent}
Let $1\leq p\leq q\leq\infty$, and let $T:E\to L$ be an operator between locally convex spaces  $E$ and $L$. Then:
 \begin{enumerate}
%\item[{\rm(i)}] the adjoint map $T^\ast: \big( L',V^\ast_{(p,q)}(L',L)\big) \to \big(E', V^\ast_{(p,q)}(E',E)\big)$ is continuous;
%\item[{\rm(ii)}] the adjoint map $T^\ast: L'_\beta \to \big(E', V^\ast_{p}(E',E)\big)$ is $p$-convergent;
\item[{\rm(i)}] the adjoint map $T^\ast: \big( L',L_{(p,q)}(L',L)\big) \to \big(E', L_{(p,q)}(E',E)\big)$ is continuous;
\item[{\rm(ii)}] the adjoint map $T^\ast: L'_{w^\ast} \to \big(E', L_{(p,q)}(E',E)\big)$ is $p$-convergent.
\end{enumerate}
\end{proposition}

\begin{proof}
(i) To show that $T^\ast$ is continuous, let $A^\circ$ be a standard $L_{(p,q)}(E',E)$-neighborhood of zero, where $A$ is a $(p,q)$-limited set in $E$. Then, by Lemma \ref{l:limited-set-1}, $T(A)$ is a $(p,q)$-limited set in $L$. Then for every $\eta\in T(A)^\circ$ and each $a\in A$, we have $|\langle T^\ast(\eta),a\rangle|= |\langle \eta,T(a)\rangle|\leq 1$ and hence $T^\ast\big(T(A)^\circ\big) \subseteq A^\circ$. Thus $T^\ast$ is continuous.

(ii) Since $L_{(p,q)}(E',E)\subseteq L_{(p,\infty)}(E',E)$ by (vi) of Lemma \ref{l:limited-set-1}, it suffices to consider the case $q=\infty$.  Let $\{\chi_n\}_{n\in\w}$ be a weakly $p$-summable sequence in $L'_{w^\ast}$. To show that $T^\ast(\chi_n)\to 0$ in $\big(E', L_{(p,\infty)}(E',E)\big)$, fix an arbitrary $B\in\mathsf{L}_{(p,\infty)}(E)$. By Lemma \ref{l:limited-set-1}, $T(B)$ is a  $(p,\infty)$-limited set in $L$ and hence
\[
\lim_{n\to\infty}\sup_{b\in B} \big|\langle T^\ast (\chi_n), b\rangle\big|= \lim_{n\to\infty}\sup_{b\in B} \big|\langle\chi_n, T(b)\rangle\big| = 0 .
\]
Therefore $T^\ast (\chi_n)\in B^\circ$ for all sufficiently large $n\in\w$. Since $B$ was arbitrary this means that $T^\ast (\chi_n)\to 0$ in $L_{(p,\infty)}(E',E)$, as desired.\qed
\end{proof}

Below we give a complete answer to the problem posed in the introduction for the $(p,\infty)$-case (namely, characterize those operators $T$ which map {\em  all  bounded } sets into $(p,q)$-limited sets or into coarse $p$-limited  sets). We are interested in this special case because it is dually connected with $p$-convergent operators, see in particular  Theorem \ref{t:bounded-to-p-lim} and Theorem \ref{t:qp*-convergent-Limited} for the case $q=\infty$.

One can naturally also ask when  the topology $L_{(p,q)}(E',E)$ in (ii) of Proposition \ref{p:T*-p-convergent} can be replaced by the strong topology $\beta(E',E)$. We answer this question in the next theorem. % which is similar to Theorem 14.1 of \cite{Gab-Pel}.
For $1\leq p<\infty$, it generalizes a characterization of operators $T$ between Banach spaces for which $T^\ast$ is $p$-convergent, see  Ghenciu \cite{Ghenciu-pGP}.
Following Definition 13.8 of \cite{Gab-Pel}, if $p\in[1,\infty]$, a locally convex space $E$ is called {\em weakly sequentially $p$-complete} if every weakly $p$-Cauchy sequence is weakly $p$-convergent.

\begin{theorem} \label{t:bounded-to-p-lim}
Let $p\in[1,\infty]$, and let $T:E\to L$ be an operator between  locally convex spaces  $E$ and $L$. Then the following assertions are equivalent:
\begin{enumerate}
\item[{\rm(i)}] for every $B\in\Bo(E)$, the image $T(B)$ is a $(p,\infty)$-limited set in $L$;
\item[{\rm(ii)}] $T^\ast: L'_{w^\ast}  \to E'_\beta$  is $p$-convergent.
\end{enumerate}
If $L'_{w^\ast} $ is sequentially complete and $T^\ast:L'_{w^\ast} \to \big(E'_\beta\big)_w$ is sequentially continuous, then {\em(i)} and {\em(ii)} are equivalent to
\begin{enumerate}
\item[{\rm(iii)}] $T^\ast\circ S$ is a sequentially precompact operator for any operator $S:\ell_{p^\ast} \to L'_{w^\ast} $ $($or $S:c_0 \to L'_{w^\ast} $ if $p=1$$)$.
\end{enumerate}
If $1< p<\infty$, $L'_{w^\ast} $ is sequentially complete  and $T^\ast:L'_{w^\ast} \to \big(E'_\beta\big)_w$ is sequentially continuous, then {\rm(i)-(iii)} are equivalent to the following
\begin{enumerate}
\item[{\rm(iv)}] $T^\ast\circ S$ is a sequentially compact operator for any operator $S:\ell_{p^\ast} \to L'_{w^\ast} $.
\end{enumerate}
If $p=1$, $E'_\beta$ and $L'_{w^\ast} $ are sequentially complete,  $T^\ast:L'_{w^\ast} \to \big(E'_\beta\big)_w$ is sequentially continuous and $L'_{w^\ast}$ is weakly sequentially $1$-complete, then {\rm(i)-(iii)} are equivalent to the following
\begin{enumerate}
\item[{\rm(v)}] $T^\ast\circ S$ is a sequentially compact operator for any operator $S:c_0 \to L'_{w^\ast} $.
\end{enumerate}
\end{theorem}

\begin{proof}
(i)$\Ra$(ii) Let $\{\chi_n\}_{n\in\w}$ be a weak$^\ast$ $p$-summable sequence in $L'$. To show that $T^\ast(\chi_n)\to 0$ in $E'_\beta$, fix an arbitrary $B\in\Bo(E)$. Since $T(B)$ is a  $(p,\infty)$-limited set in $L$ we have
\[
\sup_{b\in B} \big|\langle T^\ast (\chi_n), b\rangle\big|= \sup_{b\in B} \big|\langle\chi_n, T(b)\rangle\big| \to 0 \; \mbox{ as }\; n\to\infty,
\]
and hence $T^\ast (\chi_n)\in B^\circ$ for all sufficiently large $n\in\w$. Since $B$ was arbitrary this means that $T^\ast (\chi_n)\to 0$ in $E'_\beta$, as desired.

(ii)$\Ra$(i) Let $B\in\Bo(E)$. To show that $T(B)$ is  a  $(p,\infty)$-limited set in $L$, take any weakly $p$-summable sequence $\{\chi_n\}_{n\in\w}$ in $L'_{w^\ast}$. For every $\e>0$, the polar $\e B^\circ =\big(\tfrac{1}{\e}B\big)^\circ$ is a neighborhood of zero in $E'_\beta$. Since $T^\ast$ is  $p$-convergent, we have $T^\ast(\chi_n)\to 0$ in $E'_\beta$ and hence there is $N_\e\in\w$ such that $T^\ast(\chi_n) \in \e B^\circ$ for all $n\geq N_\e$. Therefore
\[
\sup_{b\in B} \big|\langle\chi_n, T(b)\rangle\big| = \sup_{b\in B} \big|\langle T^\ast (\chi_n), b\rangle\big| \leq \e  \; \mbox{ for all }\; n\geq N_\e.
\]
As $\e$ was arbitrary it follows that $\sup_{b\in B} \big|\langle\chi_n, T(b)\rangle\big|\to 0$. Thus  $T(B)$ is  a  $(p,\infty)$-limited set.
\smallskip

The equivalences (ii)$\LRa$(iii) and (ii)$\LRa$(iii)$\LRa$(iv)$\LRa$(v) immediately follow from Theorem 13.17 of \cite{Gab-Pel} %\ref{t:p-convergent-1}
applied to $E_1=L'_{w^\ast}$, $L_1=E'_\beta$ and $T_1=T^\ast$.\qed
\end{proof}

\begin{remark} \label{rem:Lw-Eb-cont} {\em
The condition on $T$ to be such that $T^\ast:L'_{w^\ast} \to E'_\beta$ is continuous is sufficiently strong. It is satisfied if $E$ is a feral space because $E'_\beta = E'_{w^\ast}$ and hence $T^\ast$ is automatically continuous by Theorem 8.10.5 of \cite{NaB}. Recall that an lcs $E$ is {\em feral} if every bounded subset of $E$ is finite-dimensional.\qed}
\end{remark}

Theorem \ref{t:bounded-to-p-lim} applied to the identity map $T=\Id_{E}:  E \to E$ immediately implies the following characterization of spaces for which every bounded subset is a $(p,\infty)$-limited set. Recall that a locally convex space $E$ is called {\em Grothendieck} or has the {\em Grothendieck property} if the identity map $\Id_{E'}:  E'_{w^\ast} \to \big(E'_\beta\big)_w$ is sequentially continuous.

\begin{corollary} \label{c:Bo=Lim-2}
Let  $p\in[1,\infty]$,  and let $E$ be a locally convex space. Then the following conditions are equivalent:
\begin{enumerate}
\item[{\rm(i)}] every bounded subset of $E$ is a $(p,\infty)$-limited set $\big($i.e., $\Bo(E)=\mathsf{L}_{(p,\infty)}(E)$$\big)$;
\item[{\rm(ii)}] the identity map $\Id_{E'}: E'_{w^\ast}  \to E'_\beta$  is $p$-convergent.
\end{enumerate}
If $E'_{w^\ast}$ is sequentially complete and  $E$ has the Grothendieck property, then {\em(i)}-{\em(ii)} are equivalent to
\begin{enumerate}
\item[{\rm(iii)}] any operator $S:\ell_{p^\ast} \to E'_{w^\ast}$ $($or $S:c_0 \to E'_{w^\ast}$ if $p=1$$)$ is sequentially precompact.
\end{enumerate}
If $1< p<\infty$,  $E'_{w^\ast}$ is sequentially complete and  $E$ has the Grothendieck property, then {\rm(i)-(iii)} are equivalent to the following
\begin{enumerate}
\item[{\rm(iv)}] any operator $S:\ell_{p^\ast} \to E'_{w^\ast}$  is sequentially compact.
\end{enumerate}
If $p=1$, $E'_{w^\ast}$ is a sequentially complete, weakly sequentially $1$-complete, Grothendieck space and $E'_\beta$ is sequentially complete, then {\rm(i)-(iii)} are equivalent to the following
\begin{enumerate}
\item[{\rm(v)}] any operator $S:c_0 \to E'_{w^\ast}$ is sequentially compact.
\end{enumerate}
\end{corollary}

Applying Corollary \ref{c:Bo=Lim-2} for $p=\infty$, we obtain the following assertion.
\begin{corollary} \label{c:Bo=Lim-infty}
Let $E$ be a locally convex space. Then the following conditions are equivalent:
\begin{enumerate}
\item[{\rm(i)}] every bounded subset of $E$ is limited $\big($i.e., $\Bo(E)=\mathsf{L}(E)$$\big)$;
\item[{\rm(ii)}] the identity map $\Id_{E'}: E'_{w^\ast}  \to E'_\beta$  is completely continuous {\rm(=} $\infty$-convergent{\rm)}; in particular, $E$ is a Grothendieck space.
\end{enumerate}
\end{corollary}

Below we give a useful construction of operators from $\ell_1^0$ into locally convex spaces whose adjoint is $p$-convergent.

\begin{proposition} \label{p:seq-p-convergent}
Let $\{x_n\}_{n\in\w}$ be a bounded sequence in a locally convex space $(E,\tau)$, % which does not contain null-subsequences,
and let $T:\ell_1^0 \to E$ be a linear map defined by
\[
T(a_0 e_0+\cdots+a_ne_n):=a_0 x_0+\cdots+ a_n x_n \quad (n\in\w, \; a_0,\dots,a_n\in\IF).
\]
Then $T$ is continuous. Moreover, if $E$ is locally complete, then $T$ can be extended to a continuous operator from $\ell_1$ to $E$. In any case,  if $\{x_n\}_{n\in\w}$ is a $(p,\infty)$-limited set, then $T^\ast: E'_{w^\ast}\to\ell_\infty$ is $p$-convergent.
%\begin{enumerate}
%\item[{\rm(i)}] if $\{x_n\}_{n\in\w}$ is a $p$-$(V^\ast)$ set, then $T^\ast: E'_\beta\to\ell_\infty$ is $p$-convergent;
%\item[{\rm(ii)}] if $\{x_n\}_{n\in\w}$ is a $(p,\infty)$-limited set, then $T^\ast: E'_{w^\ast}\to\ell_\infty$ is $p$-convergent.
%\end{enumerate}
\end{proposition}

\begin{proof}
The continuity of $T$ and, in the case $E$ is locally complete, the existence of the extension of $T$ are proved in Proposition 14.9 of \cite{Gab-Pel}. Assume now that $A=\{x_n\}_{n\in\w}$ is a $(p,\infty)$-limited set.
To show that the adjoint linear map $T^\ast$ is $p$-convergent, let $\{\chi_n\}_{n\in\w}$ be a  weak$^\ast$ $p$-summable sequence in $E'$. Since $A$ is a $(p,\infty)$-limited set, by definition, we have $\lim_{n\to\infty} \sup_{x\in A} |\langle\chi_n,x\rangle|=0$. Therefore
\[
\| T^\ast(\chi_n)\|_{\ell_\infty}= \sup_{k\in\w} \big|\langle T^\ast(\chi_n),e_k\rangle\big|= \sup_{k\in\w} \big|\langle \chi_n,x_k\rangle\big| \leq \sup_{x\in A} \big|\langle \chi_n,x\rangle\big| \to 0\;\; \mbox{ as } \; n\to\infty.
\]
Thus $\big\{T^\ast(\chi_n)\big\}_{n\in\w}$ is a null sequence in $\ell_\infty$, and hence $T^\ast$ is $p$-convergent.\qed
\end{proof}

We know  that $(p,q)$-limited sets are bounded. In the next  theorem we give an operator characterization of those spaces $E$ in which the  $(p,q)$-limited sets have stronger topological properties than just being bounded as, for example, being weakly sequentially (pre)compact. %; this theorem extends Theorem 15 of \cite{Ghenciu-pGP}. %  and is similar to Theorem 14.10 of \cite{Gab-Pel}.

\begin{theorem} \label{t:qp*-convergent-Limited}
Let $1\leq p\leq q\leq \infty$, and let  $E$ be a locally convex space. Then the following assertions are equivalent:
\begin{enumerate}
\item[{\rm(i)}] if $L$ is a normed space and $T:L\to E$ is an operator such that $T^\ast: E'_{w^\ast} \to L'_\beta$ is $(q,p)$-convergent, then $T$ is weakly sequentially compact {\rm(}resp., sequentially compact, weakly sequentially precompact, sequentially precompact, weakly sequentially $p$-compact or  weakly sequentially $p$-precompact{\rm)};
\item[{\rm(ii)}] the same as {\rm(i)} with $L=\ell_1^0$;
\item[{\rm(iii)}] each $(p,q)$-limited subset of $E$ is relatively weakly sequentially compact {\rm(}resp.,  relatively sequentially compact, weakly sequentially precompact, sequentially precompact,  relatively  weakly sequentially $p$-compact or  weakly sequentially $p$-precompact{\rm)}.
\end{enumerate}
Moreover, if $E$ is locally complete, then {\rm(i)-(iii)} are equivalent to
\begin{enumerate}
\item[{\rm(iv)}] the same as {\rm(i)} with $L=\ell_1$.
\end{enumerate}
\end{theorem}

\begin{proof}
(i)$\Ra$(ii) and (i)$\Ra$(iv) are clear.

(ii)$\Ra$(iii) and (iv)$\Ra$(iii): Let $A$ be a $(p,q)$-limited subset of $E$. Fix an arbitrary sequence $S=\{x_n\}_{n\in\w}$ in $A$, so $S$ is a bounded subset of $E$. Therefore, by Proposition \ref{p:seq-p-convergent}, the linear map $T:\ell_1^0 \to E$ (or $T:\ell_1 \to E$ if $E$ is locally complete) defined by
\[
T(a_0 e_0+\cdots+a_ne_n):=a_0 x_0+\cdots+ a_n x_n \quad (n\in\w, \; a_0,\dots,a_n\in\IF).
\]
is continuous. For every $n\in\w$ and each $\chi\in E'$, we have
$
\langle T^\ast(\chi),e_n\rangle=\langle \chi,T(e_n)\rangle=\langle\chi,x_n\rangle
$
and hence $T^\ast(\chi)=\big(\langle\chi,x_n\rangle\big)_n\in\ell_\infty$. In particular, $\|T^\ast(\chi)\|_{\ell_\infty}=\sup_{n\in\w} |\langle\chi,x_n\rangle|$.

Let now $\{\chi_n\}_{n\in\w}$ be a weak$^\ast$ $p$-summable sequence in $E'_{w^\ast}$. Since $A$ and hence also $S$ are $(p,q)$-limited sets we obtain $\big(\|T^\ast(\chi_n)\|_{\ell_\infty}\big)=\big(\sup_{i\in\w} |\langle\chi_n,x_i\rangle|\big)\in\ell_q$ (or $\in c_0$ if $q=\infty$). Therefore $T^\ast$ is $(q,p)$-convergent, and hence, by (ii) or (iv), the operator $T$ belongs to the corresponding class described in (i). Therefore $S=\{T(e_n)\}_{n\in\w}$ has a weakly convergent (resp., convergent, weakly Cauchy, Cauchy, weakly $p$-convergent, or weakly $p$-Cauchy) subsequence, as desired.
\smallskip

(iii)$\Ra$(i)  Let $T:L\to E$ be an operator from a normed space such that $T^\ast: E'_{w^\ast} \to L'_\beta$ is a $(q,p)$-convergent operator. Then, by Theorem \ref{t:qp-convergent-pq-limited-1}, $T(B_L)$ is a $(p,q)$-limited set and hence it is relatively weakly sequentially compact (resp.,  relatively sequentially compact, weakly sequentially precompact, sequentially precompact,  relatively  weakly sequentially $p$-compact or  weakly sequentially $p$-precompact). Thus $T$ belongs to the corresponding class described in (i).\qed
\end{proof}

The definition of coarse $p$-limited sets allows to reformulate Theorem 14.16 of \cite{Gab-Pel} as follows.

\begin{theorem}  \label{t:p-V*-set}
Let $1<p<\infty$, and let $E$ be a quasibarrelled space such that $E'_\beta$ is an $\ell_\infty$-$V_p$-barrelled space. Then the class of $p$-$(V^\ast)$ sets in $E$ coincides with the class of coarse $p$-limited sets.
\end{theorem}

In Theorem \ref{t:p-V*-set} the condition on $E$ being a quasibarrelled space is essential as Example \ref{exa:coarse-p-lim-not-pV*} below shows. First we prove the next simple lemma.
\begin{lemma} \label{l:weak-to-norm}
Let $E$ be a locally convex space such that $E=E_w$, and let $L$ be a normed space. Then every $T\in\LL(E,L)$ is finite-dimensional.
\end{lemma}

\begin{proof}
Observe that $T$ can be extended to an operator ${\bar T}$ from a completion ${\bar E}$ of $E$ to a completion ${\bar L}$ of $L$. As $E$ carries its weak topology, we obtain ${\bar E}=\IF^\kappa$ for some cardinal $\kappa$. Since ${\bar T}$ is continuous, there is a finite subset $\lambda$ of $\kappa$ such that ${\bar T}\big(\{0\}^\lambda \times \IF^{\kappa\SM \lambda} \big)$ is contained in the unit ball $B_{{\bar L}}$ of ${\bar L}$. Taking into account that $B_{{\bar L}}$ contains no non-trivial linear subspaces we obtain that $\{0\}^\lambda \times \IF^{\kappa\SM \lambda}$ is contained in the kernel $\ker({\bar T})$ of ${\bar T}$. Therefore ${\bar T}[\IF^\kappa]={\bar T}[\IF^\lambda]$ is finite-dimensional. Thus also $T$ is finite-dimensional.\qed
\end{proof}

\begin{example} \label{exa:coarse-p-lim-not-pV*}
Let $1< q \leq p<\infty$. Then the space $E:=(\ell_q)_w$ satisfies the following conditions:
\begin{enumerate}
\item[{\rm(i)}] $E$ is not quasibarrelled, but $E'_\beta=\ell_{q^\ast}$ is a Banach space;
\item[{\rm(ii)}] $B_{\ell_q}$ is a coarse $p$-limited set in $E$ which is not a $p$-$(V^\ast)$ set.
\end{enumerate}
\end{example}

\begin{proof}
The clause (i) is clear, and Lemma \ref{l:weak-to-norm} shows that $B_{\ell_q}$ is a coarse $p$-limited set in $E$. To show that $B_{\ell_q}$  is not a $p$-$(V^\ast)$ set, let $\chi_n=e^\ast_n$ for every $n\in\w$. Then the sequence $\{\chi_n\}_{n\in\w} \subseteq E'_\beta=\ell_{q^\ast}$ is weakly $p$-summable (indeed, if $(x_k)\in \ell_q=\big(\ell_{q^\ast}\big)'$, then $\big(\langle(x_k),\chi_n\rangle\big)_n=(x_n)\in \ell_q\subseteq \ell_p$). However, $\sup_{(x_k)\in B_{\ell_q}} |\langle\chi_n,(x_k)\rangle|=1\not\to 0$. Thus $B_{\ell_q}$  is not a $p$-$(V^\ast)$ set.\qed
\end{proof}
%Note that for $1<p<\infty$, Example \ref{exa:coarse-p-lim-not-pV*} is stronger than Example 1 from \cite{GalMir} because every $p$-limited set is a $p$-$(V^\ast)$ set, see (viii) of Lemma \ref{l:limited-set-1}.

Below for an important case which includes all strict $(LF)$-spaces, we characterize coarse $1$-limited sets. First we recall some definitions and results.
Following \cite{Gabr-free-resp}, a sequence $A=\{ a_n\}_{n\in\w}$ in an lcs $E$ is said to be {\em equivalent to the standard unit basis $\{ e_n: n\in\w\}$ of $\ell_1$} if there exists a linear topological isomorphism $R$ from $\cspn(A)$ onto a subspace of $\ell_1$ such that $R(a_n)=e_n$ for every $n\in\w$ (we do not assume that the closure $\cspn(A)$ of the  $\spn(A)$ of $A$ is complete or that $R$ is onto). We shall say also that $A$ is an {\em $\ell_1$-sequence}.
Following \cite{GKKLP}, a locally convex space $E$ is said to have the {\em Rosenthal property} if every bounded sequence in $E$ has a subsequence which either (1) is Cauchy in the weak topology, or (2) is equivalent to the unit basis of $\ell_1$. The following remarkable extension of the celebrated Rosenthal $\ell_1$-theorem  was proved by Ruess \cite{ruess}.

\begin{theorem}\label{ruess}
Every locally complete locally convex space $E$ whose every separable bounded set is metrizable has the Rosenthal property.
\end{theorem}
Note that for every $\ell_1$-sequence $A=\{ a_n\}_{n\in\w}$ in $E$  in Theorem \ref{ruess}, a topological isomorphism $R$ from $\overline{\spn}(A)$ to  $\ell_1$ is {\em onto}. Observe also that strict $(LF)$-spaces satisfy Theorem \ref{ruess}.

%If $E$ is a Banach space the next result in fact is obtained in Proposition 1.1 of \cite{Bombal}, quoting \cite{GS} and \cite{Emmanuele-V}.
\begin{theorem} \label{t:image-l1}
Let $E$ be  a locally complete space whose separable bounded sets are metrizable. Then for a bounded subset $A$ of $E$, the following assertions are equivalent:
\begin{enumerate}
%\item[{\rm(i)}] $T(A)$ is relatively compact for every Schur Banach space $X$ and each $T\in\LL(E,X)$;
\item[{\rm(i)}] $A$ is a coarse $1$-limited set; %$T(A)$ is relatively compact for every operator $T:E\to \ell_1$;
\item[{\rm(ii)}] $A$ does not contain an $\ell_1$-sequence $\{x_n\}_{n\in\w}$  such that the closed span $\overline{\spn}\{x_n\}_{n\in\w}$ is complemented in $E$.
\end{enumerate}
If in addition $E$ is barrelled, then {\rm(i)} and {\rm(ii)} are equivalent to
\begin{enumerate}
\item[{\rm(iii)}] $A$ is a $(1,\infty)$-limited set.
\end{enumerate}
\end{theorem}

\begin{proof}
%(i)$\Ra$(ii) is clear.
%
%(ii)$\Ra$(i)  Suppose for a contradiction that there are a Schur Banach space $X$ and an operator $T:E\to X$ such that $T(A)$ is not relatively compact in $X$. Since $X$ is a Banach space it follows that there exist a sequence $\{x_n\}_{n\in\w}$ in $T(A)$ and $\e>0$ such that $\|x_n-x_m\|\geq \e$. We claim that $\{x_n\}_{n\in\w}$ cannot contain a weakly Cauchy subsequence. Indeed, if $\{x_{n_k}\}_{k\in\w}$ is weakly Cauchy, then for every strictly increasing sequence $(k_i)\subseteq \w$, we obtain that $\{x_{n_{k_{i+1}}}-x_{n_{k_{i}}}\}$ is weakly null and hence norm null. Therefore, by the Schur property, $\{x_{n_k}\}_{k\in\w}$ is norm Cauchy and hence converges. Hence $\|x_{n_{k+1}}-x_{n_{k}}\|\to 0$ which is impossible. Now the claim and the Rosenthal property imply that $\{x_n\}_{n\in\w}$ contains an $\ell_1$-subsequence. Thus
%
(i)$\Ra$(ii) Assume that $E$ has only the Rosenthal property, and suppose for a contradiction that there is an $\ell_1$-sequence $\{x_n\}_{n\in\w}$ in $A$ such that $L:=\overline{\spn}\{ x_n: n\in\w\}$ is complemented in $E$. Let $S$ be a  projection from $E$ onto $L$, and let $R$ be a linear homeomorphism of $L$ onto a subspace of $\ell_1$ such that $R(x_n)=e_n$ for every $n\in\w$. Then $T:=R\circ S: E\to\ell_1$ is an operator such that $T(A)$ contains $\{ e_n: n\in\w\}$. Therefore $T(A)$ is not relatively compact in $\ell_1$. Thus $A$ is not coarse $1$-limited, a contradiction.
\smallskip

(ii)$\Ra$(i) Suppose for a contradiction that $A$ is not a coarse $1$-limited set. Then there is $T\in\LL(E,\ell_1)$ such that $T(A)$ is not relatively compact in $\ell_1$. By Theorem 1.4 of \cite{Nicolescu}, there is a sequence $\{x_n\}_{n\in\w}$ in $A$ such that the sequence  $S_0=\{T(x_n)\}_{n\in\w}$ is equivalent to the standard unit basis $\{e_n\}_{n\in\w}$ of $\ell_1$ and such that the subspace $H_0:=\overline{\spn}(S_0)$ is a complemented subspace of $\ell_1$. Let $R_0: H_0\to \ell_1$ be a linear topological isomorphism such that $R_0\big( T(x_n)\big)=e_n$ for every $n\in\w$. Since a continuous image of a weakly Cauchy sequence is weakly Cauchy, the sequence $\{x_n\}_{n\in\w}$ has no weakly Cauchy subsequences, and hence, by the Rosenthal property of $E$ (see Theorem \ref{ruess}), there is a subsequence $\{x_{n_k}\}_{k\in\w}$ of $\{x_n\}_{n\in\w}$ which is equivalent to $\{e_k\}_{k\in\w}$. Let $R: \overline{\spn}\{x_{n_k}\}_{k\in\w}\to \ell_1$ be a linear topological isomorphism such that $R\big( x_{n_k}\big)=e_k$ for every $k\in\w$. % (we use here the fact that $E$ is locally complete, see Theorem \ref{ruess}).
 Observe that the subspace $H_1:= \overline{\spn}\{T(x_{n_k})\}_{k\in\w}$ of $H_0$ satisfies the following two conditions:
\begin{enumerate}
\item[(1)] $H_1$ is complemented in $H_0$ (since $R_0$ is a topological isomorphism, $R_0\big(T(x_{n_k})\big)=e_{n_k}$ and $\cspn\{e_{n_k}\}_{k\in\w}$ is complemented in $\ell_1$), and hence $H_1$ is complemented also in $\ell_1$, and
\item[(2)] $H_1$ is topologically isomorphic to $\ell_1$ (since  $R_0$ is a topological isomorphism and $\cspn\{e_{n_k}\}_{k\in\w}$ is topologically isomorphic to $\ell_1$).
\end{enumerate}
Let $Q:\ell_1\to H_1$ be a projection (so $Q\big(T(x_{n_k})\big)=T(x_{n_k})$ for every $k\in\w$), and let $R_1:H_1\to \ell_1$ be a linear topological isomorphism such that $R_1\big(T(x_{n_k})\big)=e_k$ for every $k\in\w$. Since
\[
R^{-1}\circ R_1\circ Q\circ T(x_{n_k})=R^{-1}\circ R_1\big(T(x_{n_k})\big)=R^{-1}(e_k)=x_{n_k} \;\; \mbox{ for every } k\in\w,
\]
it follows that $R^{-1}\circ R_1\circ Q\circ T$ is a continuous projection from $E$ onto $\overline{\spn}\{x_{n_k}\}_{k\in\w}$ and $\{x_{n_k}\}_{k\in\w}$ is equivalent to $\{e_k\}_{k\in\w}$. But this contradicts (ii).
\smallskip

(i)$\Leftrightarrow$(iii) immediately follows from Proposition \ref{p:char-1-limited}.\qed
\end{proof}

\begin{corollary} \label{c:operator-strict-LF-L1}
Let $E$ be a strict $(LF)$-space which does not contain an isomorphic copy of $\ell_1$ which is complemented in $E$. Then every bounded subset $A$ of $E$ is a coarse $1$-limited and a $(1,\infty)$-limited set.
\end{corollary}
By the classical Pitt theorem \cite[4.49]{fabian-10}, all operators  $\LL(\ell_p,\ell_1)$ ($1<p<\infty$) and $\LL(c_0,\ell_1)$ are compact. Below we generalize this result.
\begin{corollary} \label{c:operator-Banach-l1}
If $E$ is a Banach space containing no an isomorphic copy of $\ell_1$  which is complemented in $E$, then the class of all bounded subsets of $E$ coincides with the class of all coarse $1$-limited sets.  Consequently, every $T\in\LL(E,\ell_1)$ is compact.
\end{corollary}

%\begin{remark} \label{rem:compact-Banach-l1} {\em
%The condition on $E$ in Corollary \ref{c:operator-Banach-l1} is not necessary in general. Indeed, let $E=\ell_\infty$. Then, by Corollary 5.5.4 of \cite{Al-Kal}, every operator $T\in\LL(E,\ell_1)$ is compact. On the other hand, by Theorem 2.5.7 of  \cite{Al-Kal}, $E$ contains $\ell_1$-sequences $\{x_n\}_{n\in\w}$ and the subspace $X:=\cspn\{x_n\}_{n\in\w}$ cannot be complemented in $E$ since, otherwise, the identity map $\Id_X:X\to X$ would be compact.\qed }
%\end{remark}

The condition of being a barrelled space in (iii) of Theorem \ref{t:image-l1} is essential as the following example shows.
\begin{example} \label{exa:1-lim-coarse}
Let $E=(c_0)_p$ be the Banach space $c_0$ endowed with the pointwise topology induced from $\IF^\w$, and let
\[
B=\big\{ (x_n)_{n\in\w}\in E: |x_n|\leq (n+1)^2 \mbox{ for every } n\in\w\big\}.
\]
Then $B$ is a coarse $1$-limited set in $E$ which is not $(1,\infty)$-limited.
\end{example}

\begin{proof}
It is clear that $B$ is a bounded subset of $E$. Therefore, by (ii) of Example 5.4 of \cite{Gab-Pel}, % \ref{exa:c0-1-barrel},
$B$ is a coarse $p$-limited set for every $p\in[1,\infty]$. To show that $B$ is not $(1,\infty)$-limited, consider the sequence $\{\chi_n\}_{n\in\w}= \big\{ \tfrac{e^\ast_n}{(n+1)^2}\big\}_{n\in\w}$ in $E'$. In the proof of (i) of Example 5.4 of \cite{Gab-Pel}, %\ref{exa:c0-1-barrel},
we showed that  $\{\chi_n\}_{n\in\w}$ is a weak$^\ast$ $1$-summable sequence in $E'$. Since
\[
\sup_{(x_k)\in B} |\langle\chi_n, (x_k)\rangle|=1 \;\; \mbox{ for every } n\in\w,
\]
it follows that $B$ is not a $(1,\infty)$-limited set in $E$.\qed
\end{proof}

%It is worth mentioning that if $1<p<\infty$ and $E$ is a Banach space, then the class of coarse $p$-limited sets in $E$ coincides with the class of $p$-$(V^\ast)$ sets, see Theorem \ref{t:p-V*-set}. However, in general, there are coarse $p$-limited sets which are  not $p$-$(V^\ast)$ sets, see Example \ref{exa:coarse-p-lim-not-pV*}.

\begin{corollary} \label{c:Banach-p-coarse-limited}
Let $E$ be a Banach space. Then:
\begin{enumerate}
\item[{\rm(i)}] if $p=1$, then the class of coarse $1$-limited sets in $E$ coincides with the class of $(1,\infty)$-limited sets;
\item[{\rm(ii)}] if $1<p<\infty$, then the class of coarse $p$-limited sets in $E$ coincides with the class of $p$-$(V^\ast)$ sets.
\end{enumerate}
\end{corollary}

\begin{proof}
(i) follows from Theorem \ref{t:image-l1}, and (ii) follows from Theorem \ref{t:p-V*-set}.\qed
\end{proof}

%As usual, in the case $p=\infty$ we shall omit the subscript $p$ and write simply $V$-(quasi)barrelled etc. (so that $V_\infty$-(quasi)barrelled spaces are $V$-(quasi)barrelled).

%Let $1\leq p< q\leq\infty$. Diagrams \ref{equ:V*p-diag} and \ref{equ:V*p-diag} immediately
%\[
%\xymatrix{
%\mbox{$\aleph_0$-barrelled} \ar@{=>}[r] & \mbox{$V_p$-barrelled}\\
%\mbox{$\aleph_0$-quasibarrelled} \ar@{=>}[r] & \mbox{$V_p$-quasibarrelled}
%}
%\]

%Bourgain and Diestel \cite{BourDies} showed that every limited subset of a Banach space is  weakly sequentially precompact. Using the same idea Galindo and Miranda proved in Proposition 3 of \cite{GalMir} that every coarse $p$-limited set, $2\leq p<\infty$, is weakly sequentially precompact. Below we extend these results. First we prove the next lemma. Recall that a locally convex space (resp., a Banach space) $X$ is called {\em injective} if for every locally convex space (resp., a Banach space) $E$ and its (resp., closed) subspace $H$, any operator $T:H\to X$ can be extended to an operator ${\bar T}:E\to X$.

To generalize (iii) of Theorem \ref{t:limited-BD} and its extension given  in Proposition 3 of \cite{GalMir}, first we prove the next lemma.
Recall that an lcs $X$ is called {\em injective} if for every subspace $H$ of a locally convex space $E$, each operator $T:H\to X$ can be extended to an operator ${\bar T}: E\to X$.

\begin{lemma} \label{l:injective-Banach-lcs}
Every injective Banach space is also an injective locally convex space.
\end{lemma}

\begin{proof}
Let $X$ be an injective Banach space, $H$ be a subspace of a locally convex space $E$, and let $T:H\to X$ be an operator. It is well known (see for example Exercise 5.27 of \cite{fabian-10}) that there are a set $\Gamma$ and a closed subspace $Y$ of $\ell_\infty(\Gamma)$ such that $X\oplus Y=\ell_\infty(\Gamma)$. Denote by $\pi_X:\ell_\infty(\Gamma)\to X$ the canonical projection, and let $I_X:X\to \ell_\infty(\Gamma)$, $I_X(x):=(x,0)$, be the canonical embedding. Since, by Proposition 7.4.5 of \cite{Jar}, $\ell_\infty(\Gamma)$ is an injective locally convex space, the operator $I_X\circ T:H\to \ell_\infty(\Gamma)$ can be extended to an operator $\overline{I_X\circ T}:E\to \ell_\infty(\Gamma)$. Set ${\bar T}:= \pi_X\circ (\overline{I_X\circ T})$. Then ${\bar T}$ is an operator from $E$ to $X$ such that
\[
{\bar T}(h)=\pi_X\circ \big(I_X\circ T\big)(h)=\pi_X\big((T(h),0)\big)=T(h) \;\;\mbox{ for each $h\in H$}.
\]
Thus ${\bar T}$ extends $T$ and hence $X$ is an injective locally convex space.\qed
\end{proof}

\begin{theorem} \label{t:coarse-p-lim-wsc}
Let $2\leq p\leq\infty$, and let $E$ be a locally convex space with the Rosenthal property. Then every coarse $p$-limited subset of $E$ is  weakly sequentially precompact.
\end{theorem}

\begin{proof}
We consider only the case $2\leq p<\infty$ since the case $p=\infty$ can be considered analogously replacing $\ell_p$ by $c_0$.
Suppose for a contradiction that there is a coarse $p$-limited subset $A$ of $E$ which is not weakly sequentially precompact. So there is a sequence $S=\{x_n\}_{n\in\w}$ in $A$ that does not have a weakly Cauchy subsequence. By the Rosenthal property of $E$ and passing to a subsequence if needed, we can assume that $S$ is an $\ell_1$-sequence. Set $H:=\cspn(S)$ and let $P:H\to \ell_1$ be a topological isomorphism of $H$ onto a subspace of $\ell_1$ such that $P(x_n)=e_n$ for every $n\in\w$. Let $J:\ell_1\to \ell_p$, $I_1:\ell_1\to \ell_2$, and $I_2:\ell_2\to \ell_p$ be the natural inclusions, so  $J=I_2\circ I_1$. By the Grothendieck Theorem 1.13 of \cite{DJT}, the operator $I_1$ is $1$-summing. By the Ideal Property 2.4 of \cite{DJT}, $J$ is also $1$-summing, and hence, by the Inclusion Property 2.8 of \cite{DJT}, the operator $J$ is $2$-summing.
By the discussion after Corollary 2.16 of \cite{DJT}, the operator $J$ has a factorization
\[
\xymatrix{
J: \; \ell_1  \ar[r]^R  & L_\infty(\mu) \ar[r]^{J_2^\infty}  & L_2(\mu) \ar[r]^Q & \ell_p },
\]
where $\mu$ is a regular probability measure on some compact space $K$ and $J_2^\infty:L_\infty(\mu) \to L_2(\mu)$  is the natural inclusion. By Theorem 4.14 of \cite{DJT}, the Banach space $L_\infty(\mu)$ is injective. Therefore, by Lemma \ref{l:injective-Banach-lcs},  $L_\infty(\mu)$ is  an injective locally convex space. In particular, the operator $R\circ P:H\to L_\infty(\mu) $ can be extended to an operator $T_\infty:E\to L_\infty(\mu) $. Set $T:=Q\circ  J_2^\infty \circ T_\infty$. Then $T$ is an operator from $E$ to $\ell_p$ such that
\[
T(x_n)=Q\circ  J_2^\infty \circ R\circ P(x_n)=J\circ P(x_n)=e_n \quad \mbox{ for every $n\in\w$}.
\]
Since $A$ and hence also $S$ are coarse $p$-limited sets, (iii) of Lemma \ref{l:coarse-p-limited} implies that the canonical basis $\{e_n\}_{n\in\w}$ of $\ell_p$ is also a coarse $p$-limited set. Therefore $\Id_{\ell_p}\circ T (S)=\{e_n\}_{n\in\w}$ is a relatively compact subset of $\ell_p$, a contradiction.\qed
\end{proof}
It is noticed in \cite[p.~944]{GalMir} that in general Theorem \ref{t:coarse-p-lim-wsc} is not true for $p=1$ even for Banach spaces (in fact, the closed unit ball of $C([0,1])$ is a coarse $1$-limited set which is not  weakly sequentially precompact).

\begin{corollary} \label{c:Banach-lim-wsp}
\begin{enumerate}
\item[{\rm(i)}] If $2\leq p\leq\infty$ and $E$ is a locally convex space with the Rosenthal property, then every $p$-limited subsets of $E$ is weakly sequentially precompact.
\item[{\rm(ii)}] If $E$ is a Banach space and $1\leq p<\infty$, then every $p$-limited subsets of $E$ is relatively weakly {\rm(}sequentially{\rm)}  compact.
\end{enumerate}
\end{corollary}

\begin{proof}
(i) Since $p$-limited sets are coarse $p$-limited by Proposition \ref{p:pV*-coarse-p-lim}(iv), the assertion follows from Theorem \ref{t:coarse-p-lim-wsc}.

(ii) immediately follows from Theorem 17.19 of \cite{Gab-Pel} %\ref{t:Banach-Vpp}
(which states that every $(p,p)$-$(V^\ast)$ subset of $E$ is relatively weakly compact) and (viii) of Lemma \ref{l:limited-set-1}.\qed
\end{proof}
Concerning the case $p=\infty$ in (ii) of Corollary \ref{c:Banach-lim-wsp}, we note that if a Banach space $E$ does not contain an isomorphic copy of $\ell_1$, then every limited subset of $E$ is relatively weakly (sequentially)  compact, for the proof see  \cite{BourDies} (an alternative proof is given in Theorem 1.9 of \cite{Ghenciu-15}).

\begin{remark} \label{rem:limited-not-weakly-seq-precompact} {\rm
By Corollary  \ref{c:Banach-lim-wsp}, each limited subset of a Banach space is weakly sequentially precompact. It turns out that for non-Banach spaces, this very useful assertion is not true in general. Indeed, by Example 7.12 of \cite{Gab-Pel}, %{exa:V*-not-seq-precompact}
the product $\IR^\mathfrak{c}$ contains a uniformly bounded sequence $S=\{f_n\}_{n\in\w}$ which is a $(p,q)$-$(V^\ast)$ set for all $1\leq p\leq q\leq\infty$ but is not (weakly) sequentially precompact. Since $\IR^\mathfrak{c}$ is reflexive, by (viii)  of Lemma  \ref{l:limited-set-1}, $S$ is also  a $(p,q)$-limited set.\qed}
\end{remark}

\bibliographystyle{amsplain}

\end{document}